\documentclass[a4paper, 12pt]{article}

\usepackage[sort&compress]{natbib}
\bibpunct{(}{)}{;}{a}{}{,} 

\usepackage{amsthm, amsmath, amssymb, mathrsfs, multirow, url, subfigure}
\usepackage{graphicx} 
\usepackage{ifthen} 
\usepackage{amsfonts}
\usepackage[usenames]{color}
\usepackage{fullpage}
\usepackage[normalem]{ulem}
\usepackage{tcolorbox}
\usepackage{accents}
\usepackage{soul}
\usepackage[ruled,vlined]{algorithm2e}
\usepackage{soul}

\theoremstyle{plain} 
\newtheorem{thm}{Theorem}
\newtheorem{cor}{Corollary}

\newtheorem{lem}{Lemma}

\theoremstyle{definition}

\theoremstyle{remark} 

\newtheorem*{remark0}{Remark}
\newtheorem*{astep}{A--step}
\newtheorem*{pstep}{P--step}
\newtheorem*{cstep}{C--step}

\newcommand{\prob}{\mathsf{P}}

\newcommand{\lPi}{\underline{\Pi}}
\newcommand{\uPi}{\overline{\Pi}}

\newcommand{\ber}{{\sf Ber}}
\newcommand{\unif}{{\sf Unif}}
\newcommand{\nm}{{\sf N}}

\newcommand{\gam}{{\sf Gamma}}

\newcommand{\mult}{{\sf Mult}}

\newcommand{\RR}{\mathbb{R}}

\newcommand{\ZZ}{\mathbb{Z}}

\renewcommand{\S}{\mathcal{S}}

\newcommand{\plint}{\mathscr{P}}

\newcommand{\eps}{\varepsilon}

\newcommand{\iid}{\overset{\text{\tiny iid}}{\,\sim\,}}

 \newcommand{\indep}{\perp\!\!\!\!\perp}

\newcommand{\Pboot}{P_\text{\rm boot}}
\newcommand{\Gboot}{G_\text{\rm boot}}
\newcommand{\piboot}{\pi^\text{\rm boot}}
\newcommand{\hatpiboot}{\hat\pi^\text{\rm boot}}
\newcommand{\uPiboot}{\uPi^\text{\rm boot}}
\newcommand{\lPiboot}{\lPi^\text{\rm boot}}

\title{Direct and approximately valid probabilistic inference on a class of statistical functionals}

\author{Leonardo Cella\footnote{Department of Statistics, North Carolina State University; {\tt lolivei@ncsu.edu}, {\tt rgmarti3@ncsu.edu}} \quad and \quad Ryan Martin$^*$}
\date{\today}

\begin{document}
\maketitle

\begin{abstract}
Existing frameworks for probabilistic inference assume the quantity of interest is the parameter of a posited statistical model. In machine learning applications, however, often there is no statistical model/parameter; the quantity of interest is a statistical functional, a feature of the underlying distribution.  Model-based methods can only handle such problems indirectly, via marginalization from a model parameter to the real quantity of interest.  Here we develop a generalized inferential model (IM) framework for direct probabilistic uncertainty quantification on the quantity of interest.  In particular, we construct a data-dependent, bootstrap-based possibility measure for uncertainty quantification and inference.  We then prove that this new approach provides approximately valid inference in the sense that the plausibility values assigned to hypotheses about the unknowns are asymptotically well-calibrated in a frequentist sense.  Among other things, this implies that confidence regions for the underlying functional derived from our proposed IM are approximately valid.  The method is shown to perform well in key examples, including quantile regression, and in a personalized medicine application. 

\end{abstract}

\section{Introduction}
\label{S:intro}

In applications, the quantities of interest---or {\em inferential targets}---are often ``real'' in the sense that they are features of the population under investigation, known to exist and have meaning.  For example, moments and quantiles are real in the sense that all distributions have, say, a 0.85-quantile.  On the other hand, shape, concentration, tail-index, etc.~are parameters whose meaning relies on the context provided by a suitable statistical model.  Consequently, any inferences drawn about, say, a shape parameter, would be meaningless if there is no ``true'' shape parameter associated with the population in question.  It is important to realize that these issues cannot be remedied simply by ``picking a better model.''  Indeed, modern machine learning applications often require inference on unknowns that are defined as, say, minimizers of expected loss functions.  These are real in the sense above and the problems are often too complex to expect that they could be treated as parameters of an interpretable statistical model.  So taking a traditional statistical approach to this machine learning problem amounts to introducing a statistical model and treating the real quantity of interest indirectly through the parameters of the posited statistical model.  However, as the mantra goes, {\em All models are wrong}, so any inference about the real quantity of interest is immediately at risk of being rendered useless by model misspecification bias.  This is precisely the reason why machine learners often prefer to attack the real problem directly without considering a statistical model or using the associated statistical tools developed for model-based inference.  To bridge this gap, it is important that the statistical community address the problem of direct and reliable (probabilistic) uncertainty quantification about inferential targets that are not parameters of a posited statistical model.  This is the goal of the present paper.  

We do not consider ourselves ``anti-statistical model.'' There are many good reasons to work within a model-based framework, including interpretability, computational and statistical efficiency, and convenience.  Indeed, within the context of a statistical model, one has a likelihood function which can be used to make likelihood-based inference.  This includes maximum likelihood estimation, likelihood ratio testing, and the asymptotic efficiency properties that these methods enjoy.  In terms of probabilistic inference, both the Bayesian and generalized fiducial \citep{fisherfiducial, MainHaning} frameworks rely heavily on the likelihood function.\footnote{There are variations on the Bayesian framework that can allow for inference on real quantities of interest \citep[see][and the references therein]{martin.syring.chapter2022} but we will not discuss this here.}  Consequently, inference about, say, expected loss minimizers under these frameworks is necessarily {\em indirect} and would put users at risk of model misspecification bias.  The latter criticism can be at least partially remedied by making the model ``nonparametric'' in the sense of, e.g., \citet{wasserman2006all}, where the model parameter is the distribution itself (or some other infinite-dimensional object).  Even this can be handled in a probabilistic way using Bayesian nonparametrics \citep{ghosh2003bayesian, hhmw2010, ghosal.vaart.book} or generalized fiducial \citep{cuihanning2019}.  While this might address the issue of model misspecification bias, it does so by making any inference about a ``real'' inferential target even less direct---inference about an infinite-dimensional object must be made first and then an extreme marginalization to the often low-dimensional inferential target carried out.  This adds complexity and negatively affects the interpretability and computational/statistical efficiency that originally motivated the model-based approach.  

There is another interpretation of ``nonparametric'' \citep[e.g,][]{conover1971practical} that is more in line with our perspective here.  These methods allow for inference about certain inferential targets with no/minimal model assumptions---classical examples include the sign and signed-rank tests.  The advantage of these methods is that they are {\em direct} in the sense that they do not attempt anything more than to answer specific questions about the inferential target.  Unfortunately, these classical nonparametric statistical methods are tailored to very specific problems and, to our knowledge, do not readily extend to allow for probabilistic uncertainty quantification in any general or systematic way.  In this paper---an extended version of  \citet{CellaMartinBelief2021}---we aim to develop a framework which is general/flexible enough to handle modern machine learning applications while simultaneously providing (imprecise) probabilistic uncertainty quantification about the inferential target with reliability guarantees (at least approximately).

More specifically, our goal in the present paper is to develop a (generalized) inferential model, or IM, a framework for direct---and valid---probabilistic inference on statistical functionals that are not parameters of a posited statistical model. In general, an IM is mapping that takes as input the observed data, along with any other relevant information about the data-generating process, and returns as output a lower and upper probability pair to be used for quantifying uncertainty about the unknowns; see \citet{mainMartin,martinbook} for the first considerations, and \citet{MARTIN2019IJAR, imprecisefrequentist} for a modern perspective. The need for a lower and upper probability pair, instead of just a single probability like in the Bayesian and fiducial frameworks, is to ensure that inference drawn from the IM are reliable, or {\em valid}.  Further background on this is presented in Section~\ref{S:IM} below.  Despite the benefits of having provably valid probabilistic uncertainty quantification, the original Martin--Liu construction has a shortcoming: like the Bayesian and other fiducial-like frameworks, it relies on a statistical model to define the quantity of interest and characterize its relationship to the observable data.  Therefore, direct inference on quantiles or other statistical functionals would appear to be out of reach.  

To close this gap, we draw inspiration from the work presented in \citet{martin2015, MARTIN2018105} and, more recently, in \citet{CellaMartinConformal,CellaMartinISIPTA21}, towards relaxing the requirement that a connection between observable data and quantities of interest be described via a data-generating process.  The first applications of this idea focused on streamlining the IM construction, but these still rely on specification of a statistical model.  Our key observation is that, by eliminating the requirement that the user start by writing out the data-generating process, we create an opportunity to construct valid probabilistic uncertainty quantification without the specification of a statistical model.  

After some background on IMs and generalized IMs in Section~\ref{S:IM}, we turn to the problem of inference on statistical functional defined either as the minimizer of an expected loss or as the solution of an estimating equation \citep[e.g.,][]{godambe1991estimating,huber1964,huber1981robust}; see Section~\ref{ss:nonparval} for the relevant definitions. The simple quantile example mentioned above fits in this framework, as do many modern machine learning problems.  Having made this connection, it is relatively simple, at least in principle, to construct a generalized IM that is exactly valid in the sense described below.  The output of this generalized IM, defined in Section~\ref{SS:gim.naive}, takes the form of a data-dependent consonant belief/plausibility function or, equivalently, a necessity/possibility measure pair and, as consequence of the validity property, confidence regions derived from this output achieve the nominal coverage probability exactly, for all sample sizes.  Unfortunately, this simple construction is impractical because it depends on aspects of the problem that would be unknown in every real-world application.  To overcome this, in Section~\ref{SS:gim.real}, we leverage the powerful bootstrap machinery \citep[e.g.,][]{efron1979} to construct a principled approximation to the aforementioned generalized IM.  With the introduction of bootstrap, exact validity cannot be achieved, but we prove, in Section~\ref{SS:asymptotic}, that the bootstrap-based generalized IM is approximately valid in the large-sample limit. To our knowledge, this is the first general implementation of (asymptotically) valid, prior- and model-free probabilistic uncertainty quantification. Illustrations are presented in Section~\ref{S:Examples}, including, quantile regression, a classical ``model-free'' application.  We also consider, in Section~\ref{S:dtr}, an application of the proposed generalized IM approach to a relevant problem in personalized medicine, namely, {\em dynamic treatment regimes} \citep[e.g.,][]{dynamicregimes}. We conclude in Section~\ref{S:Conclusion} with a brief summary and discussion of some open problems.

\section{Background on IMs}
\label{S:IM}


Let $Z_i$, for $i=1,\ldots,n$, denote data points taking values in a space $\ZZ$, and let $Z^n = (Z_1,\ldots,Z_n) \in \ZZ^n$.  Here the space $\ZZ$ is very general, so, for example, this covers the case where $Z_i = (X_i, Y_i)$ is a predictor and response variable pair, where $X_i$ could be high-dimensional.  In this background section, we introduce a statistical model, which is a collection of probability distributions, $P_\omega^n$, for $Z^n$, indexed by a parameter $\omega \in \Omega$.  The key point is that the parameter $\omega$ determines {\em everything} about the distribution of $Z^n$.  Consequently, if the real quantity of interest is some feature, $\theta \in \Theta$, of the $Z^n$ distribution, then $\theta$ would be expressed as a function of $\omega$, i.e., $\theta=\theta(\omega)$.  For example, if the model is Gaussian, so that $\omega = (\mu, \sigma)$ is the mean and standard deviation pair, and if the inferential target $\theta$ is the 0.75-quantile, then $\theta = \mu + \sigma z_{0.75}$, where $z_{0.75}$ is the corresponding quantile of the standard normal distribution.  Regardless of what form the mapping $\omega \to \theta$ takes, inferences about $\theta$ would be obtained by applying this mapping to inferences about $\omega$. For example, if a confidence region for $\omega$ were available, then its image under the mapping $\omega \to \theta$ would be a corresponding confidence region for $\theta$.  


In this paper, inferences are based on data-dependent, probabilistic quantifications of uncertainty---or what \citet{MARTIN2019IJAR, imprecisefrequentist} refers to as an {\em inferential model} (IM).  An IM is a mapping that takes the observed data $Z^n=z^n$ and the information encoded in the statistical model to a sub-additive capacity \citep{choquet1953} defined on a collection of subsets of $\Omega$, say, the Borel $\sigma$-algebra.  Specifically, a capacity $\gamma$ is a set function that satisfies $\gamma(\varnothing) = 0$, $\gamma(\Omega)=1$, and is monotone: $A \subseteq B$ implies $\gamma(A) \leq \gamma(B)$; sub-additivity requires that $\gamma(A \cup B) \leq \gamma(A) + \gamma(B)$ whenever $A \cap B = \varnothing$. Of course, probability measures are capacities, so the familiar frameworks like Bayesian, fiducial \citep{fisherfiducial}, generalized fiducial \citep{MainHaning}, structural \citep{fraser1968structure}, and confidence distributions \citep{mainconfdist,schweder_hjort_2016} are IMs in this sense.  However, capacities are more general than ordinary probabilities, so an IM's output could also take the form of, say, a plausibility function \citep{shafer1976mathematical, denoeux2014, dempster1968a, dempster.copss}, a possibility measure \citep{dubois.prade.book}, or something else more complicated.  For the observed data $Z^n=z^n$, relative to the posited statistical model, denote the IM's capacity by $\uPi_{z^n}$.  Define its dual/conjugate as 
\[ \lPi_{z^n}(B) = 1 - \uPi_{z^n}(B^c), \quad B \subseteq \Omega, \]
and note that sub-additivity implies $\lPi_{z^n}(B) \leq \uPi_{z^n}(B)$.  For this reason, the IM's output can be referred to as a pair $(\lPi_{z^n}, \uPi_{z^n})$ of lower and upper probabilities. If the IM output is additive, not just sub-additive, then $\lPi_{z^n}$ and $\uPi_{z^n}$ are equal and we are back to the more familiar Bayesian or fiducial case.  The motivation for non-additivity will be explained below. Observe that, through the mapping $\omega \to \theta$, assertions about $\theta$ correspond to assertions about $\omega$, so we can quantify uncertainty about $\theta$ via marginalization, e.g.,  
\[ \uPi_{z^n}(\{\omega: \theta(\omega) \in A\}), \quad A \subseteq \Theta. \]
This formalizes our above description of how inferences about $\omega$ are mapped to $\theta$.  

The interpretation of the IM output is as follows.  The sets $B \subseteq \Omega$ are assertions or hypotheses about $\omega$ and $\lPi_{z^n}(B)$ and $\uPi_{z^n}(B)$ are lower and upper probabilities for the claim ``$\omega \in B$'' based on the given data $Z^n=z^n$ and the posited statistical model.  If $\lPi_{z^n}(B)$ were large, then the data $z^n$ strongly supports the claim; alternatively, if $\uPi_{z^n}(B)$ is small, then the data $z^n$ strongly contradicts the claim. For situations in between, in which $\lPi_{z^n}(B)$ and $\uPi_{z^n}(B)$ are relatively small and large, respectively, the data is not sufficiently informative to support or contradict the claim. In such situations, the data analyst ought to consider a ``don’t know'' conclusion \citep[e.g.,][]{DEMPSTER2008365} and either collect more informative data or shift focus to a less complex assertion.

The quality of an IM is determined by the reliability of inferences drawn from it, so we are concerned with the statistical properties of the IM's output, i.e., on the properties of $\uPi_{Z^n}$ as a function of data $Z^n \sim P_\omega^n$.  We focus here on the upper probability just for brevity; all of what follows could also be described in terms of the lower probability, and we show both in our presentation of the new developments in Section~\ref{S:bootIM}.  The basic idea is as follows.  Based on the interpretation of the IM output described above, erroneous inference could be made if, for example, to an assertion $B$ that happened to be true, the IM assigned small $\uPi_{Z^n}(B)$.  An IM would be unreliable if such erroneous inferences were not controllably rare, so the validity property is designed specifically to provide the control necessary to make its inferences reliable. This is completely in line with Fisher's logic behind his tests of significance \citep[][p.~42]{fisher1973}.  More formally, an IM with output $\uPi_{Z^n}$ is said to be {\em valid} if 
\begin{equation}
\label{eq:valid.p}
\sup_{\omega \in B} P^n_{\omega}\{ \uPi_{Z^n}(B) \leq \alpha\} \leq \alpha, \quad \text{for all $\alpha \in [0,1]$ and all $B \subseteq \Omega$}.
\end{equation}
In the above expression, the true $\omega$ is contained in $B$, so $B$ is a ``true'' assertion.  Then the event $\{\uPi_{Z^n}(B) \leq \alpha\}$ is potentially problematic, especially when $\alpha$ is small, as it corresponds to a case where the inferences drawn could be wrong.  However, the right-most inequality in \eqref{eq:valid.p} ensures that this potentially problematic event has an explicit and relatively small probability under the posited model.  This calibration makes it possible for the IM to avoid the ``unacceptable'' and ``systematically misleading conclusions'' that \citet{ReidandCox2015} warn us about. The validity condition \eqref{eq:valid.p} can also be compared to the (slightly weaker) {\em fundamental frequentist principle} in \citet{Walley2002ReconcilingFP}.  Of course, the validity property as stated above is relative to the posited statistical model and, therefore, if that model happens to be wrong in some sense, then property \eqref{eq:valid.p} is meaningless; it is precisely for this reason that we look to extend the IM construction and corresponding validity property beyond those idealized cases where there is a statistical model and it is assumed to be correctly specified.  

There are a number of desirable consequences of the validity property.  First, since \eqref{eq:valid.p} includes the ``for all $B \subseteq \Omega$'' clause, the validity property carries over immediately to marginal inferences on the quantity of interest $\theta$.  Second, one can readily derive statistical procedures---hypothesis tests and confidence regions---from the IM's output, and the validity property guarantees that these will control the frequentist error rates of those procedures.  More details about this will be presented in Section~\ref{S:bootIM}.  

This brings us to the motivation for considering non-additive IMs.  It turns out that {\em IMs whose output is additive cannot be valid}.  This is the so-called {\em false confidence theorem} of \citet{Ryansatellite}; see, also, \citet{MARTIN2019IJAR, martin.partial}. Demonstrations of the challenges with additive IMs, especially for marginal inference about a function $\theta=\theta(\omega)$ of the full model parameter $\omega$, can be found in \citet{Fraser2011}, \citet{imprecisefrequentist}, \citet{cunen2020}, and \citet{ryanresponsecunen}, so we will not reproduce the details here.  The point is, in order to ensure validity and to enjoy its desirable consequences, it is necessary to consider genuinely non-additive IMs.  

This begs the question: how to construct a valid IM?  The first constructions were presented in \citet{mainMartin, martinbook}, and \citet{MARTIN2019IJAR} gives a detailed overview.  The original formulation started by expressing the statistical model in terms of a functional relationship, $Z^n = a(\omega, U^n)$, between data $Z^n$, unknown parameter $\omega$, and an unobservable auxiliary variable, say, $U^n$.  This is effectively the same starting point as fiducial, but Martin and Liu's approach differs in how this auxiliary variable is handled.  At the observed value $z^n$ of $Z^n$, the expression becomes $z^n = a(\omega, u^n)$, where $u^n$ is the unobserved value of $U^n$, so the question turns to how we can quantify uncertainty about the fixed, unobserved value $u^n$.  On the one hand, a fiducial approach quantifies uncertainty about $u^n$ using the {\em a priori} distribution of $U^n$, which leads to a $z^n$-dependent probability distribution for $\omega$ that does not satisfied the validity property in \eqref{eq:valid.p}.  On the other hand, Martin and Liu argue that the epistemological status of a fixed, unobserved value of a random variable is very different from a random variable and, therefore, uncertainty about $u^n$ should be quantified with something different/more conservative than a probability distribution.  Their original proposal used random sets \citep[e.g.,][]{nguyen2006introduction, molchanov2005} to quantify uncertainty about $u^n$ but, more recently, it was recognized that quantifying uncertainty using possibility measures was a more direct route to a valid IM \citep{imposs}. Moreover, the latter construction ensures that the IM's output is also a possibility measure, and \citet{imprecisefrequentist} argued that valid IMs of this form are the most efficient.  Therefore, we will focus here on the case where the IM output takes the form of a possibility measure.  To avoid repetition, we save the details of this construction for Section~\ref{S:bootIM}---the novelty in the main results section of this paper is in dealing with the statistical model-free context, not the basic steps of the IM construction.



\section{Direct model-free probabilistic inference}
\label{S:bootIM}

\subsection{Setup}
\label{ss:nonparval}

The discussion in the previous section focused on the case where a statistical model was specified, i.e., that the data $Z^n = (Z_1,\ldots,Z_n)$ had a distribution $P_\omega^n$ indexed by a parameter $\omega \in \Omega$.  Suppose, however, that the model parameter, $\omega$, is not directly of interest.  Instead, the goal is inference on some feature $\theta$ of the underlying distribution.  Under the posited model, $\theta=\theta(\omega)$ is a function of $\omega$, and marginal inference can be carried out more or less as usual.  The main obstacle is that forcing $\theta$ to be a function of the model parameter, $\omega$, is potentially restrictive, since the model could be misspecified.  As a somewhat extreme example, suppose the quantity of interest, $\theta$, is the variance of the distribution of $Z_1$.  If we model this with a Poisson distribution having rate parameter $\omega > 0$, then the usual estimator of $\theta$ would be the sample mean, which would be a poor estimate of the variance if the distribution is not Poisson.  

To avoid the risk of model misspecification bias, we opt to proceed without specifying a model. That is, we assume $Z^n=(Z_1,\ldots,Z_n)$ consists of independent and identically distributed (iid) components with $Z_i \sim P$; the joint distribution of $Z^n$ is denoted by $P^n$.  Note that $P$ is free to be any distribution, no constraints due to dependence on a parameter $\omega$.  In this more general case, the quantity of interest $\theta=\theta(P)$ is a functional of the underlying distribution.  And since there is no restriction on $P$, there is similarly no restriction on $\theta$, hence no risk of model misspecification bias.  

So far, we have said very little about what specifically the quantity of interest, $\theta$, is.  This will be important in what follows so, to end this problem-setup section, we give some further details about the origins of $\theta$. 
\begin{itemize}
\item Start with a loss function $\ell_\vartheta(z)$ that takes pairs $(\vartheta,z) \in \Theta \times \ZZ$ to real numbers.  This loss function is a measure of the compatibility of a data point $z$ with a generic value $\vartheta$, with large values of $\ell_\vartheta(z)$ indicating less compatibility.  Then the inferential target is defined as the minimizer of the expected loss, i.e., 
\[ \theta = \arg\min_{\vartheta \in \Theta} R(\vartheta), \quad \text{where} \quad R(\vartheta) = \int \ell_\vartheta(z) \, P(dz). \]
To estimate $\theta$ based on data $z^n$ from $P$, one defines an empirical version of the risk and take $\hat\theta_{z^n}$ to be the corresponding minimizer, i.e., 
\[ \hat\theta_{z^n} = \arg\min_{\vartheta \in \Theta} R_{z^n}(\vartheta), \quad \text{where} \quad R_{z^n}(\vartheta) = \frac1n \sum_{i=1}^n \ell_\vartheta(z_i). \]
This framework is called {\em M-estimation}. 

\item Alternatively, start with a (vector-valued) function $\psi_\vartheta(z)$ and define the inferential target as the root of the expectation, i.e., $\theta$ is a solution to the (vector) equation 
\[ \Psi(\vartheta) = 0, \quad \text{where} \quad \Psi(\vartheta) = \int \psi_\vartheta(z) \, P(dz). \]
As above, to estimate $\theta$ based on data $z^n$ from $P$, define a corresponding empirical version of the expectation and take $\hat\theta_{z^n}$ to be a solution to the equation 
\[ \Psi_{z^n}(\vartheta) = 0, \quad \text{where} \quad \Psi_{z^n}(\vartheta) = \frac1n \sum_{i=1}^n \psi_\vartheta(z_i). \]
The above equation is sometimes referred to as an {\em estimating equation}, and this general framework is called {\em Z-estimation}. Some authors, including \citet[][Chap.~7]{boos.stefanski.2013}, do not distinguish between M- and Z-estimation. 
\end{itemize}

The familiar maximum likelihood framework is a special case of M- and sometimes Z-estimation.  Suppose the statistical model $P_\theta$ is indexed by $\theta$, and that $p_\theta$ is the density function.  Then $\ell_\vartheta(z) = -\log p_\vartheta(z)$ would be a loss function and the corresponding M-estimator is the maximum likelihood estimator.  Similarly, if interchange of derivatives and integrals is allowed, then $\psi_\vartheta(z) = -\frac{\partial}{\partial\vartheta} \log p_\vartheta(z)$ can be used to express the maximum likelihood estimator as a Z-estimator.  But there are many other M- and Z-estimators in the literature, this is just one very familiar case. 

In the following two subsections, we describe this paper's proposal to provide (approximately) valid and distribution-free probabilistic inference through a so-called {\em generalized IM}.  We present this in two steps, starting in Section~\ref{SS:gim.naive} with the main idea in order to develop intuition.  The real proposal and its justification comes in Section~\ref{SS:gim.real}.

\subsection{Generalized IM: basic idea}
\label{SS:gim.naive}

As discussed in Section~\ref{S:IM} above, based on the original formulation at least, to construct a valid IM for $\theta$ we would require a functional relationship, \`a la \citet{dawidstone1982}, that describes how to simulate data $Z^n$ in terms of $\theta$.  In our present context, however, this is not possible because $\theta$ is not a ``model parameter'' that determines the distribution of $Z^n$, so a different approach is needed.  Fortunately, the {\em generalized IM} construction developed in \citet{martin2015, MARTIN2018105}, which was designed to address an altogether different challenge, can be modified to suit our present needs.  

First, we will need a function that will rank generic values $\vartheta$ of $\theta$ in terms of how well they align with the data $Z^n=z^n$.  Denote this measure by $T_{z^n}(\vartheta)$.  Throughout we will assume that $\vartheta$ having smaller values of $T_{z^n}(\vartheta)$ are higher ranked in terms of how well they align with the data $z^n$.  Naturally, the definition of the $T$ function would take into consideration how the functional, $\theta$, is defined.  For example, if there was a statistical model, $P_\theta^n$, indexed by $\theta$, then a natural choice of $T$ would be 
\[ T_{z^n}(\vartheta) = -\log \bigl\{ p_\vartheta^n(z^n) \, / \, p_{\hat\theta}^n(z^n) \bigr\}, \]
where $p_\vartheta^n$ is the model's density function and $\hat\theta$ is the maximum likelihood estimator.  A similar thing could be done using profile likelihoods if $\theta$ were a function $\theta=\theta(\omega)$ of the model parameter $\omega$.  Our focus here, however, is on the situation in which there is no statistical model, and the functional $\theta$ is defined as described at the end of the previous subsection.  For the case where $\theta$ is a risk minimizer, a natural choice of $T$ is 
\[ T_{z^n}(\vartheta) = R_{z^n}(\vartheta) - R_{z^n}(\hat\theta_{z^n}), \quad \vartheta \in \Theta. \]
Similarly, for cases where $\theta$ is a solution to an estimating equation, a natural choice is 
\[ T_{z^n}(\vartheta) = n\Psi_{z^n}(\vartheta)^\top \, S_{z^n}(\vartheta)^{-1} \, \Psi_{z^n}(\vartheta), \quad \vartheta \in \Theta, \]
where $S_{z^n}(\vartheta) = n^{-1}\sum_{i=1}^n \psi_\vartheta(Z_i) \, \psi_\vartheta(Z_i)^\top$ is the empirical estimate of the covariance matrix of the random vector $\Psi_{Z^n}(\vartheta)$.  The intuition behind both of these choices is that, under certain regularity conditions, the distribution of $T_{Z^n}(\theta)$---as a function of $Z^n \sim P^n$ with $\theta=\theta(P)$---would, at least approximately, be free of any unknowns.  This is not unlike what Wilks's theorem provides in the classical setting of likelihood ratio tests.  Our proposal here does not rely on these asymptotic properties and, hence, does not require regularity conditions; see Lemma~\ref{lem:gim.valid.naive}.  However, the approximate validity result in Theorem~\ref{thm:validity} does require certain regularity, which we discuss below.  

Then the basic idea behind the original generalized IM construction was to forget about establishing a functional relationship between the full data, the quantity of interest, and an unobservable auxiliary variable.  Instead, just create a link between an appropriate summary, $T_{Z^n}(\theta)$, and an unobservable auxiliary variable, say, $U$.  For problems involving a statistical model, this can be done with $T$ the log relative likelihood summary above \citep{cahoon2019generalized1, cahoon2019generalized2}; for prediction problems, this can be done with $T$ depending on the non-conformity score used in conformal prediction \citep{CellaMartinConformal,CellaMartinISIPTA21}.  Here we propose the following association 
\begin{equation}
\label{eq:assoc}
T_{Z^n}(\theta) = G^{-1}(U), \quad U \sim Q = \unif(0,1),
\end{equation}
where $Z^n \sim P^n$, $\theta=\theta(P)$, and $G=G_P$ is the distribution function of $T_{Z^n}(\theta)$, which depends on $P$; also, ``$Q = \unif(0,1)$'' means that $Q$ is the probability measure corresponding to the uniform distribution on $[0,1]$.  Throughout, we will assume that $T$ is such that $T_{Z^n}(\theta)$ has an absolutely continuous distribution, so that $G$ is strictly increasing and the inverse is well-defined.  Next, if we set $Z^n$ equal to the observed $z^n$, then the above equation becomes 
\begin{equation}
\label{eq:assoc.fixed}
T_{z^n}(\theta) = G^{-1}(u^\star), 
\end{equation}
where $u^\star$ is a fixed value, unknown to us because $P$ and, hence, $\theta$ and $G$ are unknown. To quantify uncertainty about the unobserved $u^\star$, we introduce a possibility measure $\uPi$ defined on $[0,1]$. What is unique about a possibility measure is that the upper probabilities are determined by an ordinary point function, $\pi: [0,1] \to [0,1]$, where the maximum value 1 is attained, according to the formula 
\[ \uPi(K) = \sup_{u \in K} \pi(u), \quad \text{for all $K \subseteq [0,1]$}. \]
The function $\pi$ is called the {\em possibility contour}.  To achieve validity, we cannot choose just any possibility measure---it must be {\em consistent} with $Q = \unif(0,1)$ in the sense that 
\begin{equation}
\label{eq:compatible}
Q(K) \leq \uPi(K), \quad \text{for all measurable $K \subseteq [0,1]$}, 
\end{equation}
as described in, e.g., Definition~9 of \citet{HOSE2021133}.  This choice ensures that $Q$ is in the credal set corresponding to $\uPi$.  This properties differs from stochastic dominance \citep[e.g.,][]{denoeux2009} because the inequality holds for all events $K$, not just one-sided intervals.  Since $Q$ is one of the simplest probabilities distributions, \eqref{eq:compatible} is relatively easy to arrange.  There are several different ways this can be done, but here we insist on defining $\uPi$ based on the contour
\[ \pi(u) = 1-u, \quad u \in [0,1]. \]
That this yields a possibility measure compatible with $Q=\unif(0,1)$ is easy to see: 
\[ \uPi(K) = \sup_{u \in K} (1-u) = 1-\inf K \geq \int_K \, du = Q(K). \]
The rationale behind this choice of $\pi$ is that, since ``good'' values of $\vartheta$ are those that make $T_{z^n}(\vartheta)$ small, and small values of $T_{z^n}(\vartheta)$ correspond to small values of $u$, we want $\pi(u)$ to be large for small $u$ values.  It turns out that $\pi(u)=1-u$ as above determines the {\em maximally specific} \citep[e.g.,][]{dubois.prade.1986} possibility measure $\uPi$ that is consistent with $Q=\unif(0,1)$ in the sense above.  


Following the fiducial-style logic, which is often referred to as the {\em extension principle} in this non-additive probabilistic framework \citep[e.g.,][]{Zadeh1975}, if the possibility measure $\uPi$ with contour $\pi$ is a quantification of uncertainty about $u^\star$, then we push this through \eqref{eq:assoc.fixed} to get the data-dependent possibility measure $\uPi_{z^n}$ on $\Theta$ with contour 
\[ \pi_{z^n}(\vartheta) = \pi\bigl( G(T_{z^n}(\vartheta)) \bigr) = 1 - G(T_{z^n}(\vartheta)), \quad \theta \in \Theta. \]
That is, a generalized IM for $\theta$ under this new distribution-free framework assigns upper probabilities to assertions $A$ according to the formula 
\begin{equation}
\label{eq:gim.upi.naive}
\uPi_{z^n}(A) = \sup_{\vartheta \in A} \bigl\{ 1 - G(T_{z^n}(\vartheta)) \bigr\}, \quad A \subseteq \Theta. 
\end{equation}
The corresponding lower probability is defined by conjugacy, i.e., $\lPi_{z^n}(A) = 1 - \uPi_{z^n}(A^c)$.  

In certain applications (e.g., Section~\ref{S:dtr}), there may be select features of the inferential target $\theta$ that are also of interest.  These can be represented as $\phi=\phi(\theta)$, where the notation $\phi$ is used to represent both the unknown feature and the function mapping the original inferential target to that feature.  In such cases, the extension principle can be applied again to construct a marginal IM for $\phi$ from that for $\theta$.  That is, define the possibility contour
\[ \pi_{z^n}^\phi(\varphi) = \sup_{\vartheta: \phi(\vartheta) = \varphi} \pi_{z^n}(\vartheta), \quad \varphi \in \phi(\Theta). \]
Then lower and upper probabilities for $\phi$ can be obtained as we did before for $\theta$:
\[ \uPi_{z^n}^\phi(C) = \sup_{\varphi \in C} \pi_{z^n}^\phi(\varphi) \quad \text{and} \quad \lPi_{z^n}^\phi(C) = 1 - \uPi_{z^n}^\phi(C^c), \quad C \subseteq \phi(\Theta). \]

Validity of the above-defined generalized IM, in the sense below, is a consequence of the consistency between $Q=\unif(0,1)$ and the possibility measure $\uPi$. But it is straightforward to check this property directly, as we do next.

\begin{lem}
\label{lem:gim.valid.naive}
The generalized IM above, with output determined by the possibility measure $\uPi_{Z^n}$ defined in \eqref{eq:gim.upi.naive} is valid in the sense that, for all $\alpha \in [0,1]$ and all $A \subseteq \Theta$, the two equivalent conditions hold:
\[ \sup_{P: \theta(P) \in A} P^n \bigl\{ \uPi_{Z^n}(A) \leq \alpha \bigr\} \leq \alpha \quad \text{and} \quad \sup_{P: \theta(P) \not\in A} P^n\{ \lPi_{Z^n}(A) > 1-\alpha\} \leq \alpha. \]
In particular, $\pi_{Z^n}(\theta(P)) \sim \unif(0,1)$ as a function of $Z^n \sim P^n$.
\end{lem}

\begin{proof}
We give the proof for the claim in terms of the upper probability; the lower probability claim follows from this and conjugacy.  Fix $P$ and let $\theta=\theta(P)$. For any $A$ that contains $\theta$, monotonicity implies that 
\[ \uPi_{Z^n}(A) \geq \uPi_{Z^n}(\{\theta\}), \]
and, for the possibility measure version above, the right-hand side is simply $\pi_{Z^n}(\theta)$, i.e., the possibility contour evaluated at $\theta$.  But since $\pi_{Z^n}(\theta) = 1 - G(T_{Z^n}(\theta))$ and $G$ is the distribution function of $T_{Z^n}(\theta)$, it follows immediately that $\pi_{Z^n}(\theta) \sim \unif(0,1)$, as a function of $Z^n \sim P^n$. Therefore, 
\[ P^n\{ \uPi_{Z^n}(A) \leq \alpha \} \leq P^n\{ \pi_{Z^n}(\theta) \leq \alpha \} = \alpha, \]
and, since this holds for all $A$ and for all $P$ such that $\theta(P) \in A$, the claim follows. 
\end{proof}

We discussed above, in Section~\ref{S:IM}, the practical interpretation of the validity property.  There is another interpretation that readers familiar with the imprecise probability literature might be more comfortable with.  If one interprets the IM output as a credal set of probability distributions consistent with $\uPi_{z^n}$ in the sense of \eqref{eq:compatible}, then the upper probability version of the validity property states that the event ``{\em all the probabilities in the credal set} assign mass $\leq \alpha$ to the true assertion $A$'' is controllably rare.  This interpretation also helps to explain why, despite the use of possibility measures, etc., we can still claim that IMs offer ``probabilistic'' uncertainty quantification. 

The following corollary gives an important consequence of the generalized IM's validity property.  That is, denote the $\alpha$ level sets of the possibility contour as 
\begin{equation}
\label{eq:gim.conf.naive}
{\cal P}_\alpha(z^n) = \{\vartheta \in \Theta: \pi_{z^n}(\vartheta) > \alpha\}, \quad \alpha \in [0,1]. 
\end{equation}
We refer to these as $100(1-\alpha)$\% plausibility regions for $\theta$, i.e., these are collections of ``sufficiently plausible'' values of $\theta$ based on data $z^n$.  Validity implies that these are also nominal confidence regions.  

\begin{cor}
The generalized IM's plausibility regions in \eqref{eq:gim.conf.naive} are nominal confidence regions in the sense that 
\[ \sup_P P^n\bigl\{ {\cal P}_\alpha(Z^n) \not\ni \theta(P) \bigr\} \leq \alpha, \quad \alpha \in [0,1]. \]
\end{cor}

\begin{proof}
Fix $P$ and let $\theta=\theta(P)$.  Then it is easy to see that ${\cal P}_\alpha(Z^n) \not\ni \theta$ if and only if $\pi_{Z^n}(\theta) \leq \alpha$.  Then the claim follows immediately from Lemma~\ref{lem:gim.valid.naive}. 
\end{proof}

\begin{remark0}
For an IM whose output is a possibility measure, a stronger notion of validity can be established.  Indeed, virtually the same proof as that above shows 
\begin{equation}
\label{eq:valid.uniform}
\sup_P P^n\{ \text{$\uPi_{Z^n}(A) \leq \alpha$ for some $A \ni \theta(P)$} \} \leq \alpha, \quad \alpha \in [0,1]. 
\end{equation}
To be clear, ``for some $A \ni \theta(P)$'' corresponds to a union\footnote{In general, this is an uncountable union, so its measurability would not be automatic.  But it can be readily seen from the argument in the proof of Lemma~\ref{lem:gim.valid.naive} that the uncountable union equals ``$\pi_{Z^n}(\theta) \leq \alpha$,'' so there is in fact no measurability issue.} over all such $A$.  Therefore, the event above is much larger event than that for a fixed $A$ that contains $\theta(P)$.  So the $\alpha$ upper bound on the probability of a larger event makes for a stronger validity conclusion.  In practice, this stronger notion of validity ensures that erroneous conclusions are controllably rare not just in ideal cases where assertions $A$ are specified in advance, but also in the more challenging scenarios where data are used to determine which assertions are to be evaluated.  This extension is possible due to the uniformity in \eqref{eq:valid.uniform} being inside the event, which means that it is a rare event that the data analyst can even find a (possibly data-dependent) true assertion to which the IM would assign small upper probability.  Since assigning small upper probability to a true assertion corresponds to a case where inference might be erroneous, the uniformity baked in to \eqref{eq:valid.uniform} provides the data analyst some additional comfort and security.
\end{remark0}


While the formulation just described is simple and achieves the desirable validity property exactly, there is one major problem: its implementation requires knowledge of the distribution function $G$.  Even if a parametric model for $P$ was known to be true, it would be completely unrealistic to expect the distribution of $T_{Z^n}(\theta)$ to be known, so this would never be the case in our present situation where no statistical model is assumed.  Next, we put forth a practical version of the generalized IM approach described above.

\subsection{Generalized IM: practical construction} 
\label{SS:gim.real}

The above formulation is deceptively simple. The obstacle hidden in that presentation is the fact that the distribution function $G$---based on the distribution of the complicated function $T_{Z^n}(\theta)$, for $Z^n \sim P^n$, with $\theta=\theta(P)$---is unavailable.  Fortunately, we can overcome this obstacle by making use the powerful {\em bootstrap} procedure developed in the seminal work by \citet{efron1979}.  The basic idea behind the bootstrap is that iid samples from the empirical distribution of the observed data $z^n$ should closely resemble iid samples from $P$. Our proposal, therefore, is to approximate the unknown distribution $G$ using this bootstrap strategy.  The details of this boostrap-based generalized IM proposal, and its (approximate) validity, are presented in this and the following subsections.

The bootstrap requires an extra level of randomization and proceeds as follows.  Our presentation below, which is based on that in \citet{kosorok2008introduction}, may look a bit different from the ``sample with replacement from the observed data'' common in the literature, but rest assured that it is the same.  Let $\xi=(\xi_1,\ldots,\xi_n)$ denote a random $n$-vector, independent of the data $Z^n$, with a multinomial distribution, $\Pboot = \mult_n(n^{-1} 1_n)$. Then we define a corresponding bootstrap version of the quantity $T_{Z^n}(\theta)$, the form of which depends on whether it is based on minimizing a risk or solving an estimating equation.  Start with bootstrap versions of the driving functions $R_{z^n}$ and $\Psi_{z^n}$:
\[ R_{z^n}^\xi(\vartheta) = \frac1n \sum_{i=1}^n \xi_i \, \ell_\vartheta(z_i) \quad \text{and} \quad \Psi_{z^n}^\xi(\vartheta) = \frac1n \sum_{i=1}^n \xi_i \, \psi_\vartheta(z_i). \]
The corresponding minimizer/root will be denoted by $\hat\theta_{z^n}^\xi$.  The intuition is that the $n$-vector $\xi$ represents the number of occurrences of each original observation in the bootstrap replicate.  The above expressions depend on the random vector $\xi$ and their distribution as a function of $\xi$ will be relevant to us here.  This distribution will be approximated in a Monte Carlo way by sampling many copies of $\xi$ from its (multinomial) distribution.  

Then the corresponding bootstrap version of $T_{z^n}(\theta)$, in the M- and Z-estimation case, respectively, is given by 
\begin{equation}
\label{eq:T.rule}
\begin{split}
T_{z^n}^\xi(\hat\theta_{z^n}) & = 
R_{z^n}^\xi(\hat\theta_{z^n}) - R_{z^n}^\xi(\hat\theta_{z^n}^\xi) \\
T_{z^n}^\xi(\hat\theta_{z^n}) & = n\Psi^\xi_{z^n}(\hat\theta_{z^n})^\top \, S_{z^n}^{\xi}(\hat\theta_{z^n})^{-1} \, \Psi^\xi_{z^n}(\hat\theta_{z^n}),
\end{split}
\end{equation}
where the matrix squeezed in the middle is $S_{z^n}^\xi(\vartheta) = n^{-1}\sum_{i=1}^n \xi_i \, \psi_\vartheta(Z_i) \, \psi_\vartheta(Z_i)^\top$.  Note the parallels between the random variables $T_{Z^n}(\theta)$ as a function of $Z^n \sim P^n$ with $\theta=\theta(P)$ fixed and $T_{z^n}^\xi(\hat\theta_{z^n})$ as a function of $\xi$ with $z^n$ fixed.  That is, in the M-estimation case, say, $\theta$ minimizes the expected loss with respect to $P$ whereas $\hat\theta_{z^n}$ minimizes the expected loss with respect to $\xi \sim \Pboot$. If we denote by $\Gboot$ the distribution function of $T_{z^n}^\xi(\hat\theta_{z^n})$ as a function of $\xi \sim \Pboot$ for fixed $z^n$, then the same argument presented in the previous subsection can be used to justify the following formula for a possibility contour:
\begin{align}
\label{eq:plausboot}
\piboot_{z^n}(\vartheta) &= 1 - \Gboot(T_{z^n}(\vartheta)) \nonumber \\
&= \Pboot\{T_{z^n}^\xi(\hat\theta_{z^n}) > T_{z^n}(\vartheta)\}, \quad \vartheta \in \Theta.
\end{align}
And since our goal is probabilistic inference about $\theta$, the theory developed in the previous subsection suggests a generalized IM whose output is a possibility measure: 
\begin{equation}
\label{eq:gim.out}
\uPiboot_{z^n}(A) = \sup_{\vartheta \in A} \piboot_{z^n}(\vartheta), \quad A \subseteq \Theta. 
\end{equation}
As before, the lower probability $\lPiboot_{z^n}$ is defined by conjugacy. 

This version of the possibility contour still is not practical, since evaluating probabilities with respect to $\Pboot$ requires a sum over all $n^n$ possible values of $\xi$.  For a practical alternative, we suggest a Monte Carlo approximation based on taking samples of $\xi$ from $\Pboot$.  That is, for a user-specified bootstrap sample size $B$, define 
\begin{equation}
\label{eq:MCplaus}
\hatpiboot_{z^n}(\vartheta) = \frac1B \sum_{b=1}^B 1\{T_{z^n}^{\xi_b}(\hat\theta_{z^n}) > T_{z^n}(\vartheta)\}, \quad \xi_b \sim \Pboot, \quad b=1,\ldots,B.
\end{equation}
For low-dimensional $\theta$, this function can be plotted to visualize our uncertainty quantification based on data $Z^n$; see Figures~\ref{fig:quantile}(a), \ref{fig:spatial_median}(b), \ref{fig:reg}(a) and \ref{fig:dynreg}(a). Moreover, the possibility measure $\uPiboot_{z^n}(A)$ can be approximated by replacing the supremum with a maximum over a user-specified grid of $\vartheta$ values. These details are summarized in Algorithm~\ref{algo:conformal}.  

If marginalization to some feature $\phi=\phi(\theta)$ of the inferential target were desired, then this can be carried out exactly as described above.  Just apply the extension principle to the possibility measure defined by the contour $\hatpiboot_{z^n}$ to get a marginal IM for $\phi$.  All the properties enjoyed by the IM for $\theta$, especially the validity property in Theorem~\ref{thm:validity}, apply equally to this marginal IM for $\phi$.  In particular, marginal plausibility regions for $\phi$ would achieve the nominal frequentist coverage probability asymptotically.  

Note that the proposed generalized IM requires very little input from the user: only the data and a specification of how the inferential target is defined, which effectively identifies $T$.  With just this essential information about the problem at hand, in particular, no statistical model specification required, the proposed method produces a full, data-dependent, probabilistic quantification of uncertainty about $\theta$, which can be used to make (approximately) valid inference; see Section~\ref{SS:asymptotic}. In Section~\ref{S:Examples}, we illustrate our proposed method with several practically relevant examples. 


\begin{algorithm}[t]
\SetAlgoLined
 initialize: data $z^n$, definition of $T_{z^n}(\cdot)$ in \eqref{eq:T.rule}, and a grid of $\vartheta$ values\; 
 \For{$b$ in $1,\ldots,B$}{
 sample $\xi^b \sim \Pboot$\;
 evaluate $T_{z^n}^{\xi_b}(\hat\theta_{z^n})$ according to \eqref{eq:T.rule}\; 
 \For{each $\vartheta$ value on the grid}{
 evaluate $\hatpiboot_{z^n}(\vartheta)$ as in \eqref{eq:MCplaus}\;
 }
 }
 return $\hatpiboot_{z^n}(\vartheta)$ for each $\vartheta$ on the grid and/or $\uPiboot_{z^n}(A) \approx \max_{\vartheta \in A \cap \text{grid}} \hatpiboot_{z^n}(\vartheta)$.
 \caption{\textbf{Direct, model-free generalized IM}}
 \label{algo:conformal}
\end{algorithm}

\subsection{Asymptotic validity}
\label{SS:asymptotic}

Here we present the (asymptotically approximate) validity property enjoyed by the proposed bootstrap-based generalized IM.  Recall that, if the distribution function $G$ were known, as in Section~\ref{SS:gim.naive}, then validity followed almost immediately, as shown in Lemma~\ref{lem:gim.valid.naive}.  So, if the bootstrap version, $\Gboot$, is an accurate approximation of $G$, then a suitable approximate validity property for the more practical bootstrap-based generalized IM will follow.  
More formally, we say that the bootstrap approximation described above of the distribution of $T_{Z^n}(\theta)$, under $Z^n \sim P^n$ with $\theta=\theta(P)$, is {\em consistent} if 
\begin{equation}
\label{eq:boot.consistent}
\sup_t \bigl| G^{(n)}(t) - \Gboot^{(n)}(t) \bigr| \to 0 \quad \text{in $P^n$-probability as $n \to \infty$}, 
\end{equation}
where, to highlight their dependence on the sample size, $G^{(n)}$ denotes the exact distribution function for $T_{Z^n}(\theta)$ and $\Gboot^{(n)}$ is its bootstrap version. While the intuition behind bootstrap consistency is clear---when $n$ is large, iid sampling from the empirical distribution of $Z^n$ should be roughly the same as iid sampling from $P$---the precise technical details are complicated and non-trivial.  Fortunately, there is a substantial body of literature on bootstrap consistency, starting with \citet{bickelfreedman81} and \citet{singh1981} and, since then, \citet{wellner.zhang.1996}, \citet{ChatterjeeBose2005}, and \citet{chenghuang2010} who deal with the general M- and Z-estimation cases; see, also, \citet{hall.book}, \citet{shao.tu.book},  \citet{vaartwellner1996}, and \citet{kosorok2008introduction} for textbook-style introductions to the bootstrap consistency theory.  The general rule of thumb is that, if the M- or Z-estimator itself is asymptotically normal, i.e., if $n^{1/2}(\hat\theta_{Z^n} - \theta)$ converges in distribution to a Gaussian limit, which is true in a wide range of applications, then the bootstrap would be consistent in the sense of \eqref{eq:boot.consistent}. 

The following theorem makes more precise our above claim that bootstrap consistency is enough to establish approximate validity of our proposed generalized IM. 

\begin{thm}
\label{thm:validity}
Suppose that the inference problem is such that the bootstrap version of the distribution of $T_{Z^n}(\theta)$, as a function of $Z^n \sim P^n$, with target $\theta=\theta(P)$, is consistent in the sense of \eqref{eq:boot.consistent}.  Then the bootstrap-based generalized IM for $\theta$, whose output is determined by the possibility measure $\uPiboot_{Z^n}$ defined in \eqref{eq:gim.out}, is approximately valid in the sense that, for all $\alpha \in [0,1]$, all $A \subseteq \Theta$, the following two equivalent properties hold:
\begin{align*}
\limsup_{n \to \infty} P^n\{ \uPiboot_{Z^n}(A) \leq \alpha\} & \leq \alpha, \quad \text{for all $P$ with $\theta(P) \in A$} \\
\limsup_{n \to \infty} P^n\{ \lPiboot_{Z^n}(A) > 1-\alpha\} & \leq \alpha, \quad \text{for all $P$ with $\theta(P) \not\in A$}. 
\end{align*}
In particular, $\piboot_{Z^n}(\theta(P)) \to \unif(0,1)$ in distribution, as $n \to \infty$, under $Z^n \sim P^n$. 
\end{thm}

\begin{proof}
Since 
\[ \pi_{Z^n}(\theta) = 1-G^{(n)}(T_{Z^n}(\theta)) \quad \text{and} \quad \piboot_{Z^n}(\theta) = 1-\Gboot^{(n)}(T_{Z^n}(\theta)), \]
we immediately get 
\begin{equation}
\label{eq:remainder}
\pi_{Z^n}^\text{boot}(\theta) = \pi_{Z^n}(\theta) + \Delta_n, 
\end{equation}
where 
\[ |\Delta_n| = \bigl| G^{(n)}(T_{Z^n}(\theta)) - \Gboot^{(n)}(T_{Z^n}(\theta)) \bigr| \leq \sup_t \bigl| G^{(n)}(t) - \Gboot^{(n)}(t) \bigr|. \]
It follows from \eqref{eq:boot.consistent} that $\Delta_n = o_P(1)$.  Since $\pi_{Z^n}(\theta) \sim \unif(0,1)$ for all $n$, from \eqref{eq:remainder} and Slutsky's theorem we get that $\piboot_{Z^n}(\theta(P)) \to \unif(0,1)$ in distribution as $n \to \infty$.  The other two properties in the theorem statement are a consequence of this. 
\end{proof}

An immediate and relevant consequence of Theorem~\ref{thm:validity} is that the bootstrap-based generalized IM's plausibility region 
\begin{equation}
\label{eq:plausint}
{\cal P}_\alpha(z^n) = \{\vartheta : \hatpiboot_{z^n}(\vartheta) > \alpha\}, \quad \alpha \in [0,1], 
\end{equation}
is also an approximate confidence region in the sense that its coverage probability is converging to the nominal level $1-\alpha$ as $n \to \infty$.  Similarly, for any $A \subset \Theta$, an asymptotic size-$\alpha$ test of a hypothesis ``$\theta(P) \in A$'' rejects the hypothesis if and only if $\uPi_{z^n}(A) \leq \alpha$. These claims also apply to the summaries---plausibility regions and tests---of the marginal IMs for features $\phi = \phi(\theta)$ derived from the IM for $\theta$.  

The theorem above is quite general, but it is not universal, i.e., there are cases when the bootstrap fails to be consistent.  These instances of bootstrap failure are associated with certain non-regularities, so our assumption \eqref{eq:boot.consistent} implicitly imposes regularity conditions on the M- or Z-estimation problem.  See Section~\ref{S:Conclusion} for more discussion about non-regular cases.  There we also address the natural question about the accuracy of the approximate validity claims in Theorem~\ref{thm:validity}.

\section{Examples}
\label{S:Examples}

The goal of the present section is to illustrate the generalized IM construction above in various practically relevant examples involving quantiles. Each of the examples follows, roughly, the same structure. We start by providing the appropriate loss function or (vector-valued) estimating equation that links the observed data to the desired quantile. We then explore the IM's basic output, the possibility contour, obtained through Algorithm~\ref{algo:conformal}, in several ways. In particular,
\begin{itemize}
    \item we plot it to visualize the plausibility of different values of the quantile of interest based on a single data set;
    \vspace{-2mm}
    \item we derive confidence regions for that quantile from it through \eqref{eq:plausint};
    \vspace{-2mm}
    \item we carry out simulation studies to confirm the IM's approximate validity, checking its behavior for a range of sample sizes.
\end{itemize}



\subsection{Quantiles}
\label{SS:quantiles}

The most intuitive example of a quantity of interest that is not most naturally defined as a model parameter is the $\tau$-th quantile, the exact point $\theta=\theta_\tau$ such that $F(\theta) = \tau$, for $\tau \in (0,1)$, where $F$ is the distribution function of a random variable $Z$. Every distribution has quantiles, but very rarely are they model parameters. Of course, one can make model-based inference on a quantile by specifying a parametric model $P_\omega$, for $\omega \in \Omega$, and defining
$\theta=\theta(\omega)$ as the corresponding quantile, but, as argued above, this creates a risk of bias due to model misspecification and/or model selection. Our approach here is model-free, so these risks/challenges are avoided. 

Suppose $Z$ has a generic distribution $P$.  Then it is well-known that a general $\tau^\text{th}$ quantile of $P$ can be defined as the minimizer of the risk function $R(\vartheta) = \int \ell_\vartheta(z) \, P(dz)$, where the loss function is given by 
\[ \ell_\vartheta(z) = \tfrac12 \bigl\{ (|z-\theta| - z) + (1-2\tau)\theta \bigr\}. \]
In the special case $\tau=0.5$, corresponding to the median, this can be reduced to 
\[ \ell_\vartheta(z) = |z - \vartheta|. \]
Consistency of the bootstrap, in the sense of \eqref{eq:boot.consistent}, is established in this quantile inference problem by both \citet{bickelfreedman81} and \citet{singh1981}. 

As an illustration, suppose that $P$ is a $\gam(4,1)$. Interest here is in the median $\theta=\theta_{0.5}$ which, in this case, is roughly equal to 3.67. Figure~\ref{fig:quantile}(a) shows the plausibility contour in \eqref{eq:MCplaus} with $B=500$ and the loss function above for a single data set $z^n$ with $n=100$. The peak is at the M-estimator, i.e., the sample median, which is close to the true median, and the horizontal line determines the corresponding 95\% plausibility interval, derived by \eqref{eq:plausint}. In order to check that approximate validity is attained, a simulation study was conducted where the above scenario is repeated 1000 times and, for each data set, $\hatpiboot_{z^n}(\theta)$ is evaluated at $\theta=3.67$. Figure~\ref{fig:quantile}(b) shows that the distribution of $\hatpiboot_{Z^n}(\theta)$ is close to $\unif(0,1)$, so approximate validity is verified. The same simulation is repeated for $\tau=0.25$ and $0.75$, showing that the approximate validity conclusion is not specific to the median. Finally, Figures~\ref{fig:quantile}(c) and (d) show that smaller sample sizes do not affect the good conclusions observed in Figure~\ref{fig:quantile}(b) too much. 

\begin{figure}[t]
\begin{center}
\subfigure[Plausibility contour in \eqref{eq:MCplaus}]{\scalebox{0.5}{\includegraphics{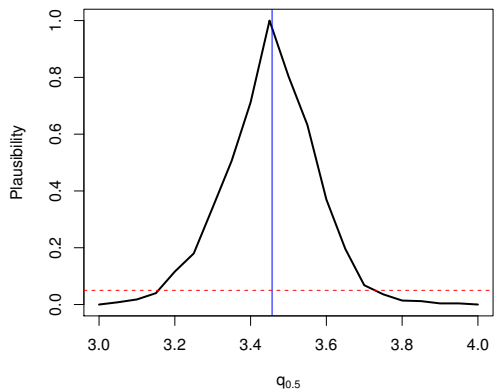}}}
\subfigure[Distribution function of $\hatpiboot_{Z^n}(\theta)$, $n=50$]{\scalebox{0.48}{\includegraphics{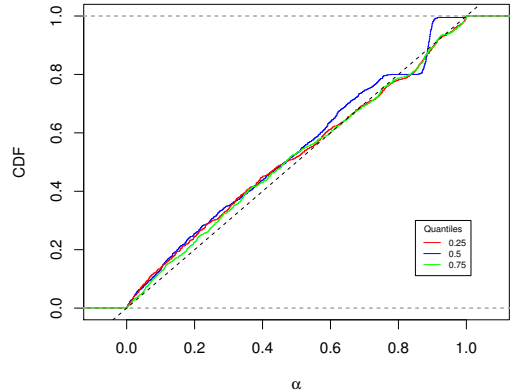}}}
\subfigure[Distribution function of $\hatpiboot_{Z^n}(\theta)$, $n=75$]{\scalebox{0.48}{\includegraphics{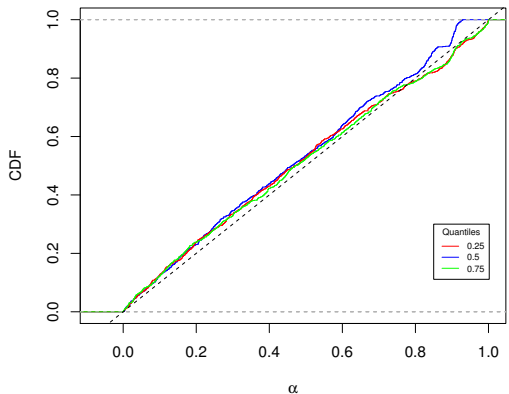}}}
\subfigure[Distribution function of $\hatpiboot_{Z^n}(\theta)$, $n=100$]{\scalebox{0.48}{\includegraphics{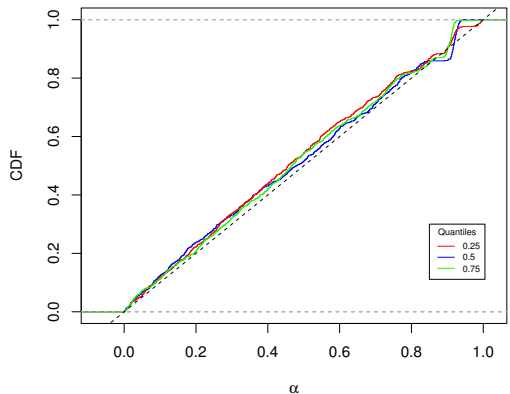}}}
\end{center}
\caption{Details from the quantile example in Section~\ref{SS:quantiles}. The results in Panel~(b), (c) and (d) are based on 1000 data replications, and shown for $\tau \in \{0.25, 0.5, 0.75\}$.}
\label{fig:quantile}
\end{figure}

There are a number of different strategies available for constructing confidence intervals for population quantiles. For further illustration, we compare our method to two of them: an exact-but-conservative solution based on the binomial distribution and a basic bootstrap procedure, which resamples the data with replacement, computes the desired quantile and then reports, for a $(1-\alpha)\%$ confidence interval, the $\frac{\alpha}{2}$ and $(1-\frac{\alpha}{2})$ quantiles of this bootstrapped distribution. For our simulation, we consider $P$ to be a Cauchy distribution with location and scale parameters equal to 2 and 1, respectively. We generated 1000 data sets of size $n=100$ and, from each, 95\% confidence intervals for $\theta_\tau$, with $\tau \in \{0.25, 0.50, 0.75\}$, based on the three methods are extracted. Table~\ref{tab:quantiles} reports the  estimated coverage probabilities and mean length of these intervals. Note that approximately validity of the generalized IM solution is confirmed.  Moreover, it is slightly more efficient than both the conservative and basic bootstrap methods.

\begin{table}[t]
\centering
\begin{tabular}{cc c c}
\hline
$\tau$ & GIM & Conservative & Bootstrap \\
\hline
$0.25$ & 0.95 (1.12) & 0.96 (1.23) & 0.96 (1.20)  \\
$0.50$ & 0.95 (0.62) & 0.98 (0.69) & 0.96 (0.64)  \\
$0.75$ & 0.93 (1.12) & 0.96 (1.25) & 0.94 (1.15)\\
\hline
\end{tabular}
\caption{Estimated coverage probabilities and mean length of 95\% confidence intervals for the quartiles based on the three following methods: generalized IM (GIM) with $B=500$; the conservative method based on the binomial distribution; the standard bootstrap method with $B=500$. The sample size is $n=100$, with data coming from a Cauchy distribution with location and scale parameters 2 and 1, respectively.}
\label{tab:quantiles}
\end{table}


\subsection{Multivariate median}
\label{SS:mult.median}

In univariate analysis, it is well known that the median is a more robust measure of the distribution's center than the mean. This is also the case in multivariate analysis. However, replacing the multivariate mean by
a multivariate median is not so straightforward. Indeed, since multivariate data do not have a natural ordering, there are various different ways of defining an order, each leading to a definition of the multivariate median or, more generally, multivariate quantiles \citep{Oja2013}. 

The most common version of a multivariate median is the {\em spatial median}.  This can be defined similar to the univariate median described above, as the minimizer of a risk function $R(\vartheta) = \int \ell_\vartheta(z) \, P(dz)$ where the loss is given by 
\[ \ell_\vartheta(z) = \|z - \vartheta\|_2 - \|z\|_2, \quad z,\vartheta \in \RR^q, \quad q \geq 1, \]
where $\|\cdot\|_2$ is the usual $\ell_2$-norm for vectors in $\RR^q$. Alternatively, the spatial median can be defined as a Z-estimator, i.e., it satisfies the system of equations $\Psi(\vartheta) = \int \psi_\vartheta(z) \, P(dz) = 0$  where $\psi_\vartheta(z)$ is a $q$-vector with components
\[\psi_{\vartheta}(z)_j = \frac{z_j - \vartheta_j}{\|z - \vartheta\|_2 }, \quad j=1,\ldots,q. \]
Consistency of the bootstrap for the multivariate median was established recently as part of Theorem~1 in \citet{bhattacharya.ghosal.2022}. 


For a quick illustration, Figure~\ref{fig:spatial_median}(a) shows the data $z_i \in \RR^2$ for $i=1,\ldots,n=200$, which are samples from bivariate normal with mean $\theta=(1,1)^\top$, unit variances, and correlation 0.7. In Figure~\ref{fig:spatial_median}(b), the plausibility contour in \eqref{eq:MCplaus} is shown, based on the $\psi_\vartheta$ function defined above, with $T$ the quadratic form in \eqref{eq:T.rule}, and $B=500$. The shaded area in Figure~\ref{fig:spatial_median}(c) represents the 95\%
plausibility region for $\theta$ derived by \eqref{eq:plausint} and, in black, the classic 95\% confidence ellipse based on the asymptotic normality. Note how the IM solution is more efficient. Figure~\ref{fig:spatial_median}(d) shows the resulting empirical distribution of the simulation study where the above scenario is repeated 1000 times and, for each data set, $\hatpiboot_{Z^n}(\theta)$ is evaluated. Approximate validity is once again verified. This is also true when a smaller sample of size $n=50$ is considered.

\begin{figure}[t]
\begin{center}
\subfigure[Scatterplot of data $z^n$]{\scalebox{0.53}{\includegraphics{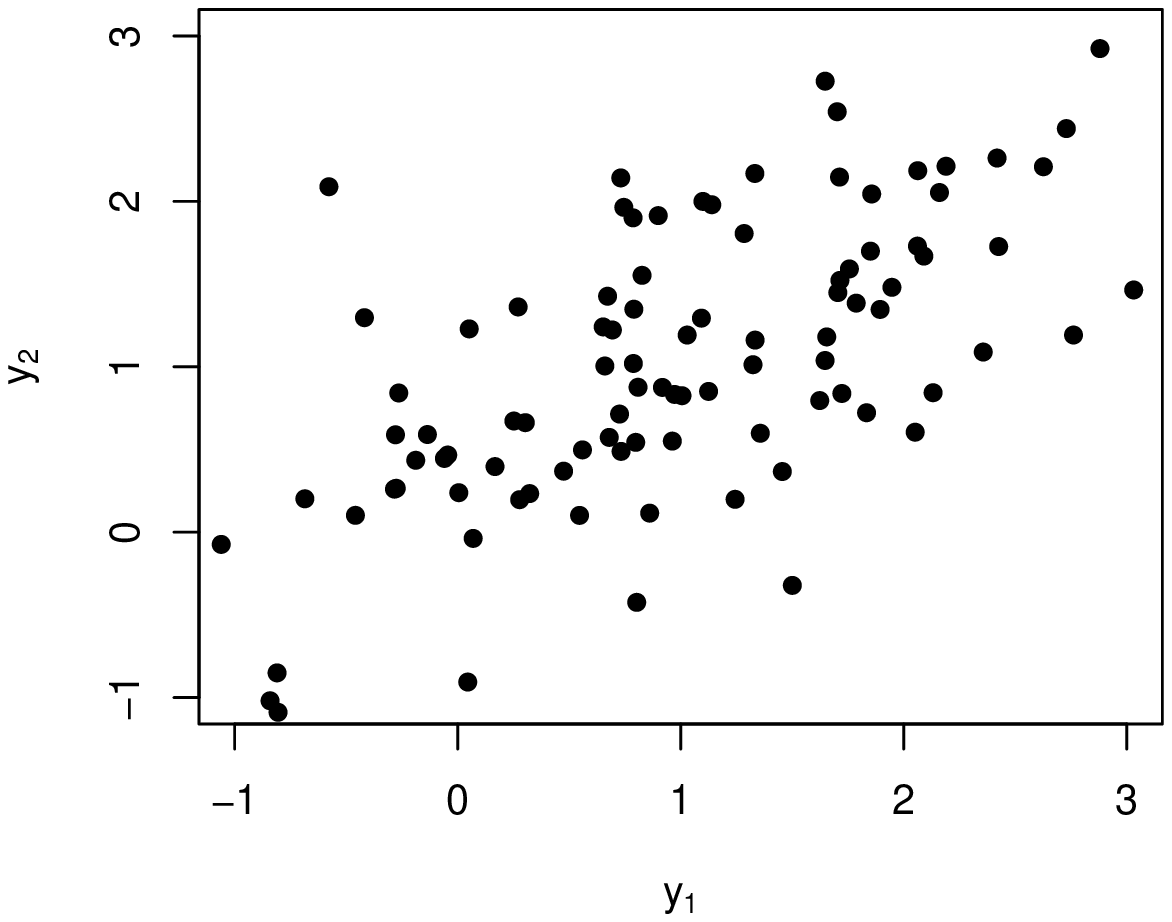}}}
\subfigure[Plausibility contour, $\vartheta \mapsto \hatpiboot_{z^n}(\vartheta)$]{\scalebox{0.54}{\includegraphics{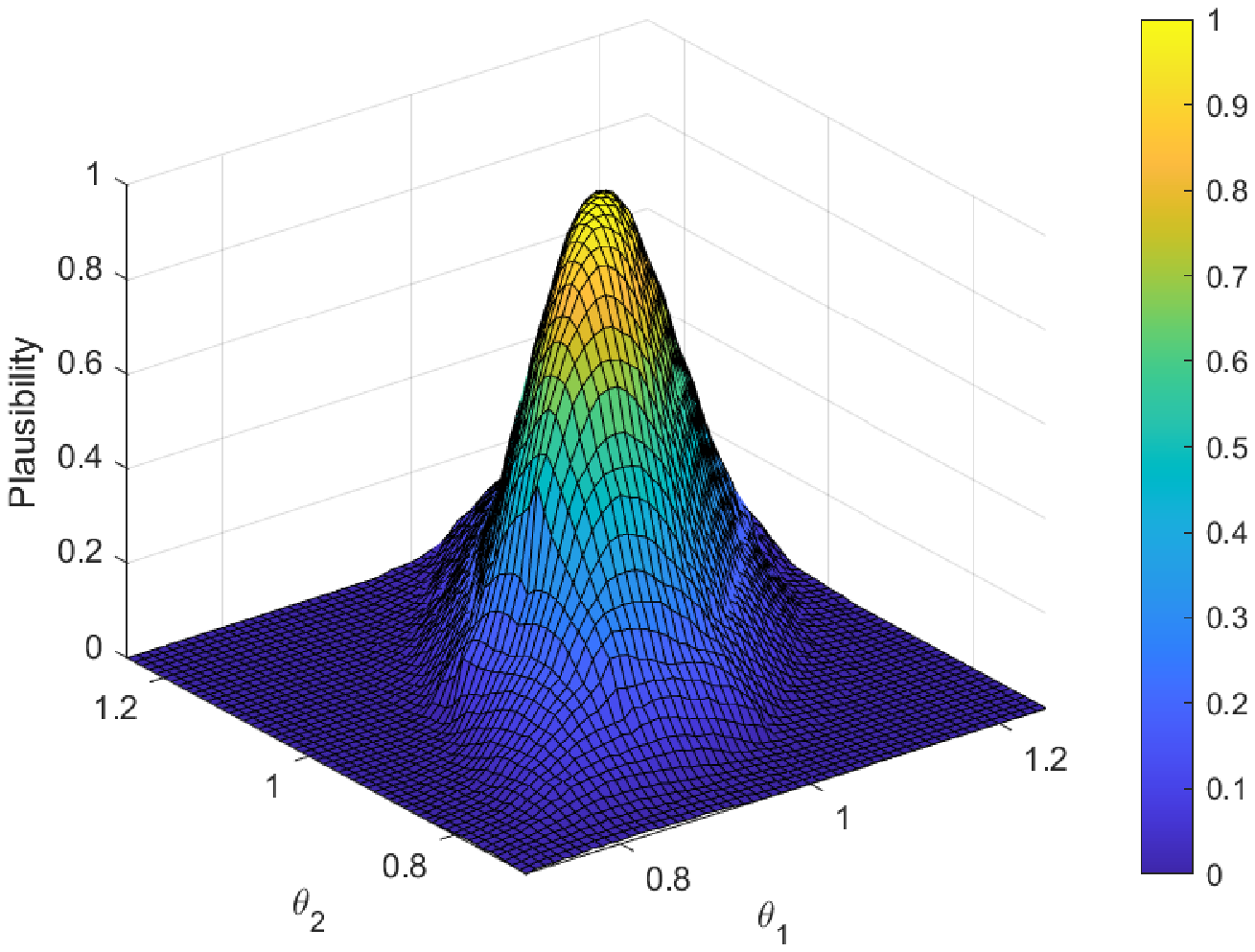}}}
\subfigure[Confidence regions for $\theta$]{\scalebox{0.57}{\includegraphics{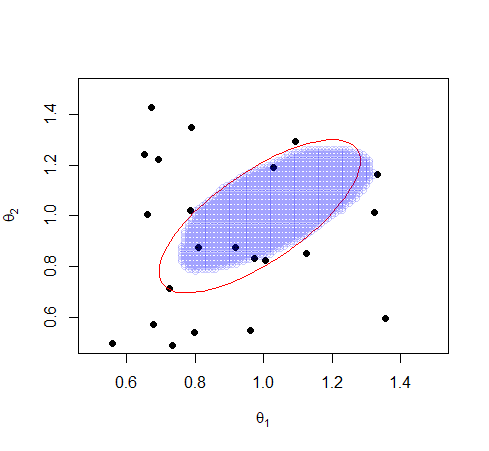}}}
\subfigure[Distribution function of $\hatpiboot_{Z^n}(\theta)$]{\scalebox{0.5}{\includegraphics{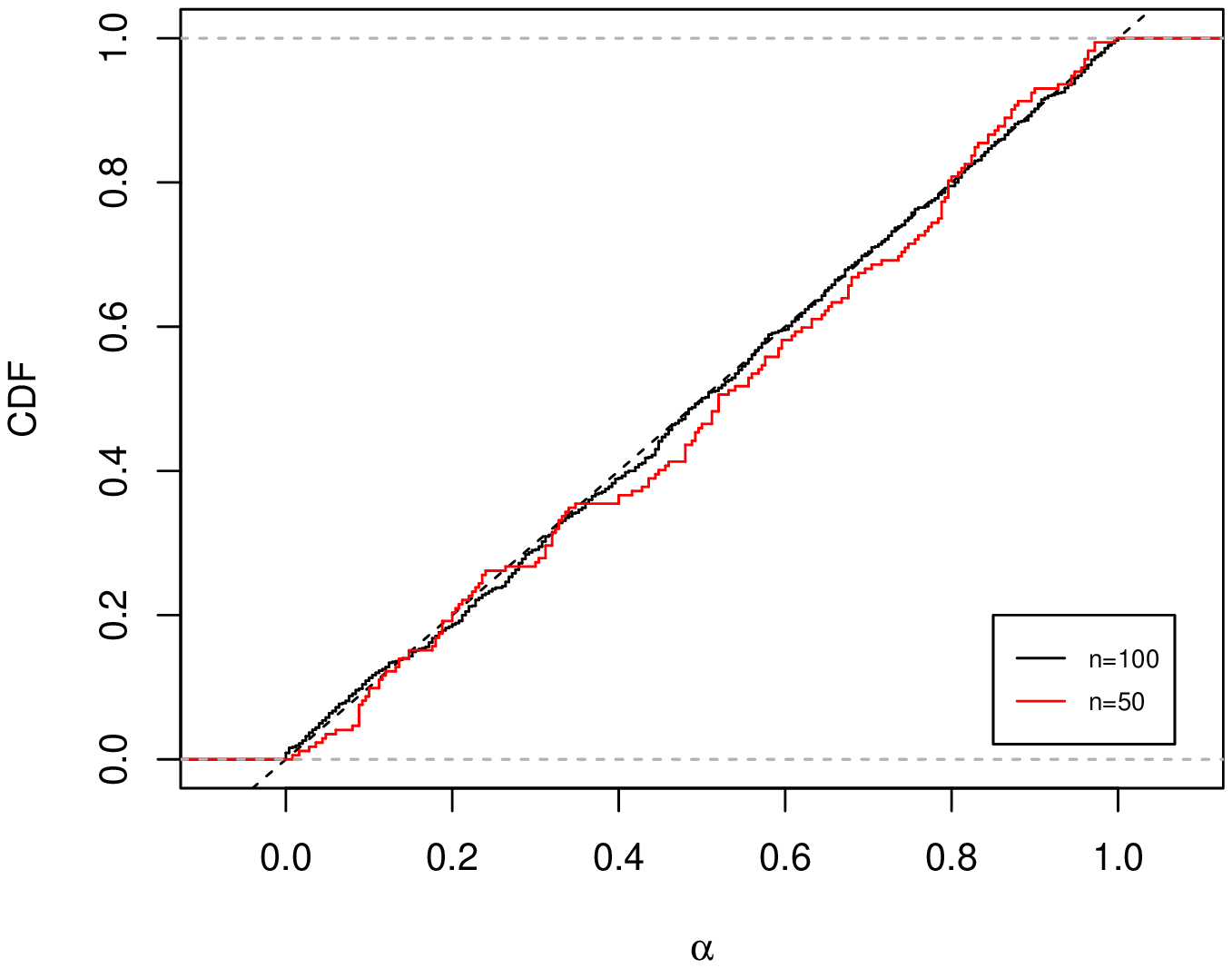}}}
\end{center}
\caption{Results for the illustration in Section~\ref{SS:mult.median}. Panel~(c): classical 95\% confidence ellipse based on asymptotic normality (red) and the 95\% plausibility region (blue).}
\label{fig:spatial_median}
\end{figure}


\subsection{Quantile regression}
\label{ss:quantile}

Let $Y$ be a response variable coupled with a covariate $X \in \RR^p$. The goal of {\em quantile regression} is to estimate the quantile for the conditional distribution of $Y$, given $X$. Fix a probability $\tau \in (0,1)$ and let $Q_x(\tau)$ denote the $\tau^\text{th}$ conditional quantile of $Y$, given $X=x$.  Then the quantile regression model says 
\[Q_x(\tau) = x^\top \theta,\]
where $\theta=\theta_\tau \in \RR^p$ is the vector of regression coefficients of interest. This ``model'' describes the functional form of the quantile function, but does not determine the distribution of $Y$, given $X=x$.  Towards making inference on $\theta$, \citet{KoenkerHallock1978} show that $\theta$ is a risk-minimizer with respect to the loss function
\[\ell_{\vartheta}(x,y) = |y - x^\top \vartheta| - (2\tau-1)x^\top \vartheta.\] 
Properties of the quantile regression M-estimator, e.g., its asymptotic normality, are investigated in \citet{koenker2005} and, in particular, bootstrap consistency is established in \citet{hahn1995}.  Here we present an illustration of the proposed generalized IM solution to the quantile regression problem.  

Let $X_i \iid \unif(0,4)$, $i=1,\ldots,n$, with $n=200$, and let $Y_i = \mu(X_i) +  \eps(X_i)$, where $\mu(x) = 4 + 0.1 x$, and $\eps(x) \sim \nm\bigl(0,(0.1 + 0.1x)^2\bigr)$. Suppose the interest is $\theta=\theta_\tau$ for $\tau=0.75$.  Figure~\ref{fig:reg}(b) displays  the data, the estimated quantile regression line corresponding to the Z-estimator $\hat\theta_{z^n}$. Plausibility contours are obtained for $\theta$ based on the loss function above and $B=500$, and the plot shows
the marginal plausibility contours for $\mu$ at selected values of $x$. The corresponding 95\% marginal plausibility band for $\mu$ is shown in Figure~\ref{fig:reg}(a).
Approximate validity of the plausibility bands is implied by the approximate validity of the generalized IM.  To check this claim empirically, we simulate 1000 data sets according the above scheme and calculated $\hatpiboot_{Z^n}(\theta)$, in each replication. Figure~\ref{fig:reg}(c) shows the empirical distribution of these values over replications and it is clear this closely matches a uniform distribution, confirming Theorem~\ref{thm:validity}.  The same plots for $\tau=0.25$ and $\tau=0.50$ are included and all suggest the uniform approximation for the distribution of $\hatpiboot_{Z^n}(\theta)$ is accurate across a range of quantile levels. Figure~\ref{fig:reg}(d) is analogous to Figure~\ref{fig:reg}(c), but considering $n=100$. As in the previous examples, the proposed solution achieves approximately validity even in finite-sample scenarios.

\begin{figure}[t]
\begin{center}
\subfigure[Data and fitted quantile curve]{\scalebox{0.57}{\includegraphics{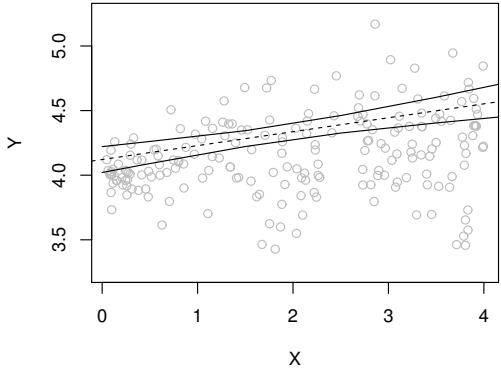}}}
\subfigure[Plausibility contours on top of Panel~(a)]{\scalebox{0.56}{\includegraphics{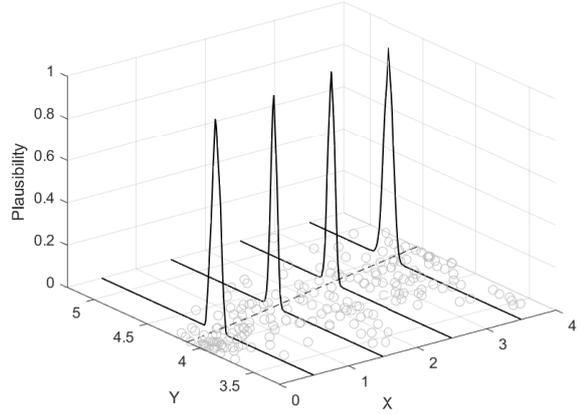}}}
\subfigure[Distribution function of $\hatpiboot_{Z^n}(\theta_\tau)$, $n=200$]{\scalebox{0.5}{\includegraphics{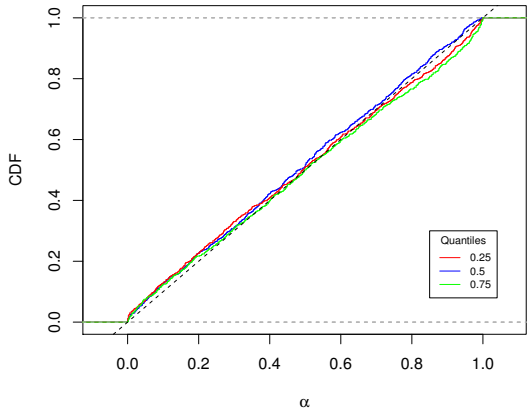}}}
\subfigure[Distribution function of $\hatpiboot_{Z^n}(\theta_\tau)$, $n=100$]{\scalebox{0.5}{\includegraphics{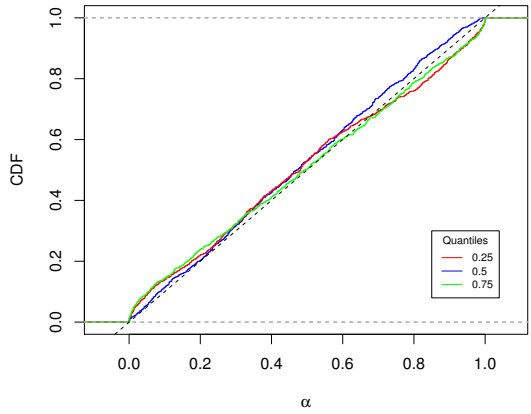}}}
\end{center}
\caption{Results for the illustration in Section~\ref{ss:quantile}. Panel~(a) shows the fitted quantile curve, for $\tau=0.75$, with the 95\% plausibility band.
}
\label{fig:reg}
\end{figure}

\section{Dynamic treatment regimes}
\label{S:dtr}

\subsection{Introduction}
Compared to the common one-size-fits-all approach in medicine, where treatment decisions are developed for the ``average'' patient, precision medicine focuses on tailoring treatment decisions to individual patients based on certain characteristics of their profile. {\em Dynamic treatment regimes} provide a formal precision medicine framework, where the individualization of treatments is dictated by a sequence of decision rules, one per stage of intervention, that are based on the patient’s
``history,'' which includes both covariates and past treatments \citep{chakradynreg}. 

In this section, we aim to provide a generalized IM solution to relevant problems that arise when considering dynamic treatment regimes, or regimes for short. For example, given a specific regime, a first basic problem is quantifying uncertainty about the expected outcome if the population under study were to receive treatment according to its rules. This expected outcome is also referred to as the {\em value} of a regime. In Section~\ref{ss:valueregime}, we construct a generalized IM for this purpose. A more challenging problem is where there is a large class of candidate decision rules and the goal is to identify an optimal one, i.e., the regime that maximizes the expected benefit to patients based on their history.  In Section~\ref{ss:optregime} we make the notion of an optimal regime precise, and develop a generalized IM to quantify uncertainty about its value.

Before tackling the above relevant problems, we provide a short background on dynamic treatment regimes and setup the basic notation. The details presented throughout this Section are largely based on \citet{dynamicregimes}.



\subsection{Background and notation}
\label{ss:backgroundregime}
For simplicity, our focus throughout this Section will be on the so called {\em single decision} problem, where there is only one stage at which a treatment must be selected from among a given set of available options.\footnote{We adopt the convention in \citet{dynamicregimes} that any treatment regime, single or multistage, whose decision rules potentially vary according to baseline and evolving patient information, is dynamic. However, other authors \citep[e.g.,][]{murphy2021} consider single decision regimes as ``non-dynamic.''} In this situation, a dynamic treatment regime consists of a single rule that takes as input the available patient information and returns as output one of the treatment options. Denote this set of possible treatment action options as $\mathcal{A}$. 
Formally, a dynamic treatment regime is defined as a {\em decision rule} $d(x)$, a function that maps an individual's covariates to a treatment option in $\mathcal{A}$, that is, $d: \mathcal{X} \to \mathcal{A}$, where $\mathcal{X}$ is the support of the covariates $X$. 
Here, we explore only the simplest case where $\mathcal{A}$ contains two treatment options, e.g., a control/active treatment scenario, so $d$ takes $X \in \mathcal{X}$ as input and returns either 0 or 1. 

Let $Z_i = (X_i,A_i,Y_i)$, for $i=1,\ldots,n$, represent the observed data from $n$ patients, where $A_i, Y_i$ and $X_i$ denote, respectively, the treatment received (either 0 or 1), the observed outcome under treatment $A_i$, and the covariates collected, for the $i^\text{th}$ patient. 
Central to both the definition of the value of a given regime and the notion of the optimal regime, to be explored, respectively, in Sections~\ref{ss:valueregime} and \ref{ss:optregime} below, is the concept of {\em potential outcome} for any regime $d \in \mathcal{D}$. 
Generally speaking, a potential outcome \citep[e.g.,][]{rubin74,rubin2005} is the outcome for an individual under a potential treatment. In our context, the random variables $Y^*(0)$ and $Y^*(1)$ represent the outcome that would be achieved by a randomly chosen individual with covariates $X$ in the population of interest if she were to receive treatment 0 or 1, respectively. Note that potential outcomes are hypothetical constructs, since a patient receives only one of the treatments, not both. The idea is to consider what this outcome would have been had the patient received the other treatment option. Now, if treatment is assigned according to regime $d$, the potential outcome of a patient is defined as
\begin{equation}\label{eq:potout}
Y^*(d) = Y^*(1) \, d(X) + Y^*(0) \, \{1 - d(X)\}.
\end{equation}

We end this subsection by pointing out that the essential results to be explored in the upcoming subsections depend on three fundamental assumptions that are common in the causal inference literature. 
\begin{itemize}
    \item {\em Stable unit treatment value assumption:} the outcome $Y$ of a patient who received treatment $A$ is the same as her potential outcome for that treatment, i.e.,
    \[ Y = Y^*(1) \, A + Y^*(0) \, (1-A).\]
    \item {\em No unmeasured confounders assumption:} all of the information used to make treatment decisions is captured by the observed covariates $X$, so that
    \[  \bigl[ \{ Y^*(0), Y^*(1) \} \indep A \bigr] \mid X. \]
    \item {\em Positivity assumption:} For any $X=x$, there are individuals receiving both treatment options, that is, $P(A=a \mid X=x)>0$ for $a=0,1$. 
\end{itemize}

\subsection{Value of a regime}
\label{ss:valueregime}
\subsubsection{Overview}
When considering a specific regime $d \in \mathcal{D}$, a fundamental question is how its use in the entire population would affect the outcome of interest, on average. With the definition of a potential outcomes in \eqref{eq:potout}, the value of any regime $d \in \mathcal{D}$ is defined as 
\[ \mathcal{V}(d) = E\{Y^*(d)\}.\] 

Towards uncertainty quantification of $\mathcal{V}(d)$ based on observed data $Z^n=z^n$, the challenge is deducing the distribution of $Y^*(d)$, which depends on that of $(X,Y^*(1),Y^*(0))$, from the distribution of the observable $(X,A,Y)$. Under the assumptions stated in the end of Section~\ref{ss:backgroundregime}, it can be shown that
\begin{equation}
\label{eq:valuereg}
E\{Y^*(d)\} = E[E(Y \mid X, A=1) \, d(X) + E[E(Y \mid X, A=0) \, \{1 - d(X)\}],
\end{equation}
where the outer expectation is with respect to the marginal distribution of $X$. If we introduce an outcome regression relationship---or {\em Q-function}---for the conditional mean, 
\begin{equation}\label{eq:outreg}
Q_{x,a}(\theta) = E(Y \mid X=x,A=a),
\end{equation}
depending on a parameter $\theta$; see \eqref{eq:Qlinear}. Then \eqref{eq:valuereg} becomes
\begin{equation}
\label{eq:valuereg*}
E\{Y^*(d)\} = E[Q_{X,1}(\theta) \, d(X) + Q_{X,0}(\theta) \, \{1 - d(X)\}].
\end{equation}

\subsubsection{Generalized IM construction}
\label{sss:valueregime2}

First, the simple connection between the Q-function---which depends on a parameter $\theta$---and the mean of the response $Y$ allows for a straightforward IM construction for $\theta$ interpreted as a risk minimizer.  For a given form describing the Q-function's dependence on the parameter $\theta$, we can define a loss function as 
\[ \ell_\vartheta(z) = \{y - Q_{x,a}(\vartheta)\}^2, \quad z=(x,a,y). \]
Then the generalized IM construction for the minimizer of the expected loss proceeds exactly as in, say, the quantile regression application above.  Asymptotic normality of the corresponding M-estimator and consistency of the bootstrap hold for very general Q-function specifications.  Note that this generalized IM construction for inference on the risk minimizer $\theta$ does not require that the posited functional form of $Q$ to be correct.  That is, the existence of the risk minimizer does not require that $Q_{x,a}(\theta)$ be the {\em true} conditional mean of $Y$, given $X=x$ and $A=a$; moreover, as we discussed in Section~\ref{S:intro}, if the risk minimizer exists, the it is a ``real'' inferential target, so drawing inference on the risk minimizer is meaningful whether there is a correctly-specfied model or not. 

We are not primarily interested in the aforementioned risk minimizer since, typically, the inferential target is some other characteristic of the problem.  Fortunately, these other characteristics can often be expressed as functions of $\theta$, i.e., as features $\phi=\phi(\theta)$; we will consider two such features below.  From the previously-described generalized IM for $\theta$, uncertainty quantification about the value of a given regime is readily obtained through marginalization as discussed in Section~\ref{SS:gim.naive}.  There is a catch, however, that deserves to be emphasized: these features have their desired interpretation only as functions of {\em the parameter $\theta$ on which the true $Q$-function depends}.  So, in order for marginal inference about these particular features, derived from a generalized IM for $\theta$, to be meaningful, it is required to assume that the posited form of the Q-function is correctly specified; that is, $\theta$ is not just the risk minimizer but determines the true conditional mean function.  Note, however, that this does not require correct specification of a statistical model---which includes distributional forms---for the observable data $Z_i = (X_i, A_i, Y_i)$.  The situation we are describing here falls under the general umbrella of {\em semiparametric inference}, where only a part of the model is assumed to be correctly specified.  

We present the details of our proposed generalized IM in the context of an example.  Consider the simulated observational study presented in \url{https://laber-labs.com/dtr-book/booktoc.html},\footnote{This is \citet{dynamicregimes}'s companion website, where several examples are provided. The particular example we explore here can be found under the ``Chapter 3'' tab.} whose objective is to assess the effectiveness of a fictitious medication developed for the treatment of hypertension. 
Each patient in this study either received the new medication ($A=1$) or received no treatment ($A=0$) based on patient/physician discretion. The outcome of interest $Y$ is the change in systolic blood pressure (mmHg) after six months of treatment, i.e., $Y = Y_0 - Y_6$. The covariates $X = (X_1,X_2)$ are, respectively, the total cholesterol (mg/dl) and the potassium level (mg/dl). Here is how the data are generated. Let $Y_{0,i} \iid \nm(160,12^2)$, $i=1,\ldots,n$, with $n=1000$ and the constraint $140 < Y_{0,i} \leq 200$. Let $X_{1,i} \iid \nm(211,45^2)$, $X_{2,i} \iid \nm(4.2,0.35^2)$, $A_i \sim \ber(\pi(x_{1,i},y_{0,i}))$, where
\[\pi(x,y)  = \frac{\exp\{-16.348 + 0.078 y + 0.017 x\}}{1 + \exp\{-16.348 + 0.078 y + 0.017 x\}},\]
and $Y_{6,i} = Y_{0,i} - \nm(\mu(x_i,a_i),3^2)$, where
\[\mu(x,a) = -15 -0.2 x_1 + 12 x_2 +a(-65 +0.5 x_1 - 5.5 x_2).\]
This implies an outcome regression relationship $Q_{x,a}(\theta)$ in \eqref{eq:outreg} given by 
\begin{equation}
\label{eq:Qlinear}
Q_{x,a}(\theta) =  \theta_0 +\theta_1 x_1 + \theta_2 x_2 +  \theta_3 a  +  \theta_4 a x_1 + \theta_5 a x_2, 
\end{equation}
depending on $\theta = (\theta_0,\theta_1,\theta_2,\theta_3, \theta_4, \theta_5)$. From the definition of $\mu(x,a)$ above, the true values of $\theta_0$, $\theta_1$, $\theta_2$, $\theta_3$, $\theta_4$, and $\theta_5$ are $-15$, $-0.2$, $12$, $-65$, $0.5$ and $-5.5$, respectively.  Note that the only aspect of the above description that the generalized IM assumes as ``true'' is the Q-function specification in \eqref{eq:Qlinear}; the statements concerning the {\em distributions} of the observables are not used at all in the generalized IM formulation nor are they assumed true in the supporting theory presented in Section~\ref{SS:asymptotic}. 

As a first check, we empirically verify the approximate validity claim in Theorem~\ref{thm:validity} for inference on $\theta$, the true parameters of the regression function.  Figure~\ref{fig:dtr.valid} shows the distribution function of $\hatpiboot_{Z^n}(\theta)$ based on repeated sampling from the data-generating process described above, with $\theta$ the true values.  As the theory predicts, we see that the distribution of $\hatpiboot_{Z^n}(\theta)$ is almost exactly $\unif(0,1)$, which means that inference drawn on $\theta$ in this setting is approximately---and almost exactly---valid. 

\begin{figure}[t]
\begin{center}
\scalebox{0.7}{\includegraphics{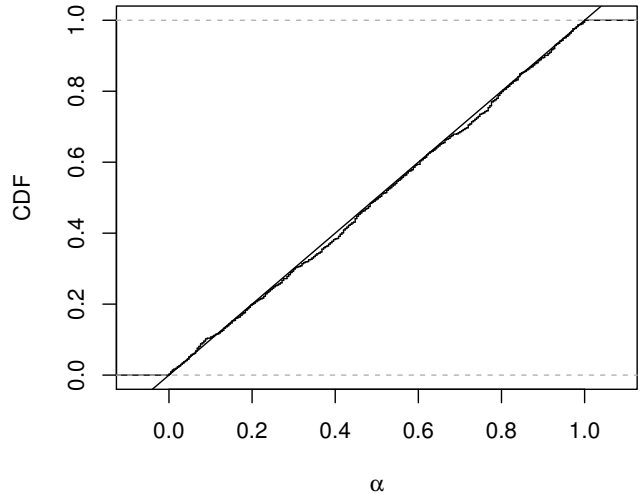}}
\end{center}
\caption{The plot shows the empirical distribution function of $\hatpiboot_{Z^n}(\theta)$ under repeated sampling from the data-generating process described in the text.}
\label{fig:dtr.valid}
\end{figure}

Next, we consider three different marginalization examples.  First, in treatment/control scenarios, it is common that the first enquiry performed by the data analyst concerns the presence of a {\em treatment effect}. Here the treatment effect is the average change in systolic blood pressure, after six months, if everyone in the population took the new medication compared to if everyone took the old medication. For the Q-function in \eqref{eq:Qlinear}, the treatment effect is $\phi := \theta_3 + \theta_4 \, E(X_1) + \theta_5 \, E(X_2)$, so interest is in the assertion ``$\phi = 0$.''  For the particular data $z^n$, the upper probability assigned to this assertion is approximately 0; this is not particularly surprising, given that the true treatment effect is equal to $-65 + 0.5\times 211 - 5.5\times 4.2= 17.4$ mmHg.


Second, we investigate marginal inference on the value of a fixed regime.  In this case, we consider two such regimes: 
\begin{itemize}
    \item the {\em static} regime where all individuals are recommended to receive the new medication, i.e., $d(X) \equiv 1$;
\vspace{-2mm}
    \item the {\em covariate-dependent} regime that assigns patients to receive the treatment if their cholesterol level exceeds a certain threshold, i.e., 
\begin{equation}
\label{eq:not.static}
d(X) = 1\{X_1>120 \ \text{mg/dl}\}.
\end{equation}
\end{itemize}
From \eqref{eq:valuereg*}, the values of the static and covariate dependent regimes are, respectively,
\begin{align*}
\phi_\text{\sc st} &:=E\{\theta_0 + \theta_3 + (\theta_1 + \theta_4)X_1 + (\theta_2 + \theta_5)X_2 \} \\
\phi_\text{\sc cd} &:=E[\theta_0 + \theta_1 X_1 + \theta_2 X_2 + (\theta_3 + \theta_4X_1 + \theta_5X_2)1\{X_1>120\} ],
\end{align*}
and Figure~\ref{fig:dynreg}(a) shows the corresponding marginal plausibility contours for each, both obtained from the generalized IM for $\theta$.  The plot suggests that, not unexpectedly, the covariate-dependent regime is no worse than the static regime. 

Third, a relevant question in practice might be whether  there is a difference between two fixed regimes. For example, given the overlap of the  plausibility contours in Figure~\ref{fig:dynreg}(a), one may wonder if the covariate-dependent regime is in fact better than the static regime. More specifically, interest is whether $\phi_\text{\sc cd} > \phi_\text{\sc st}$. Figure~\ref{fig:dynreg}(b) shows the marginal plausibility contour for the difference $\phi_\text{\sc cd} - \phi_\text{\sc st}$, obtained, once again, from the generalized IM for $\theta$. Despite the differences being small, the assertion ``$\phi_\text{\sc cd} = \phi_\text{\sc st}$'' has zero plausibility, confirming the superiority of the covariate-dependent regime. Whether these small differences are practically relevant is a separate question.


\begin{figure}[t]
\begin{center}
\subfigure[Plausibility contour for $\phi = {\cal V}(d)$]{\scalebox{0.55}{\includegraphics{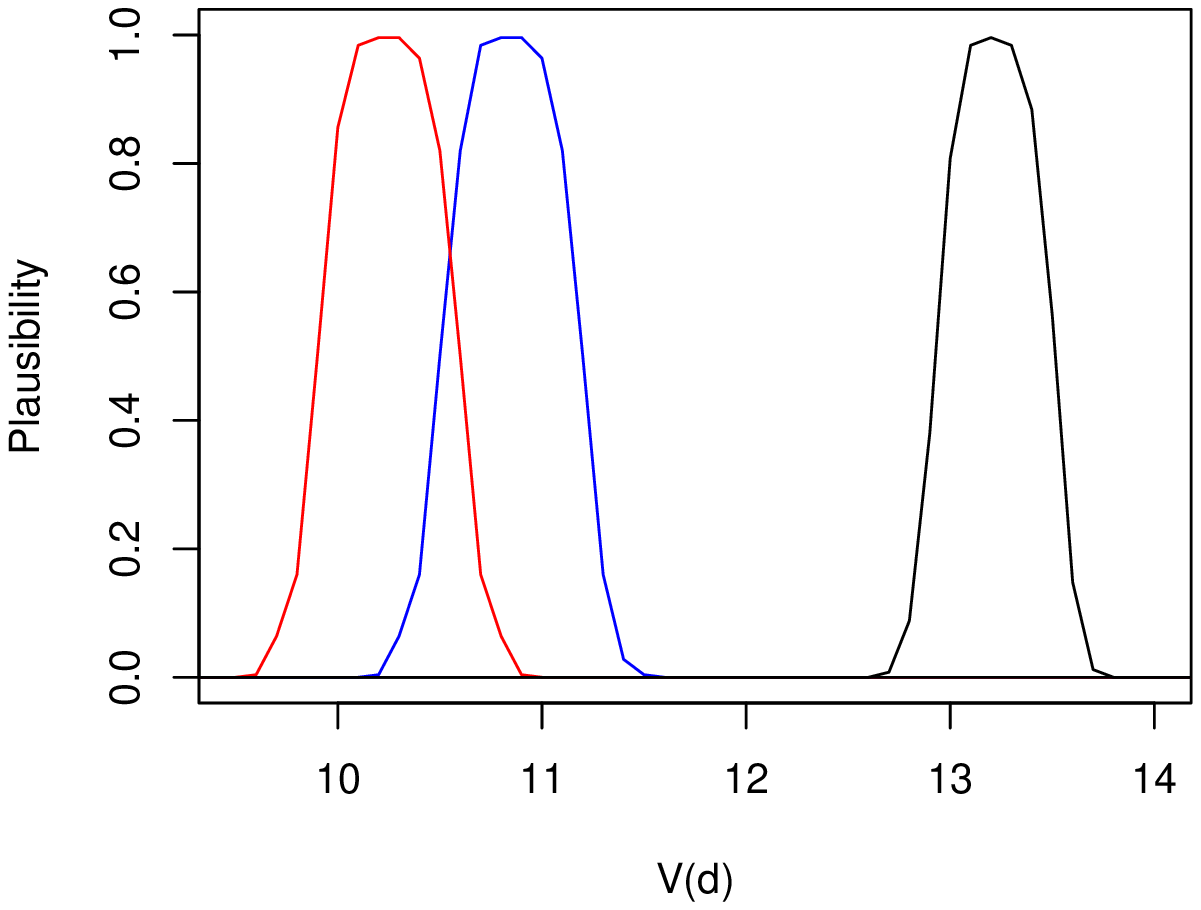}}}
\subfigure[Plausibility contour for $\phi_\text{\sc cd}-\phi_\text{\sc st}$]{\scalebox{0.55}{\includegraphics{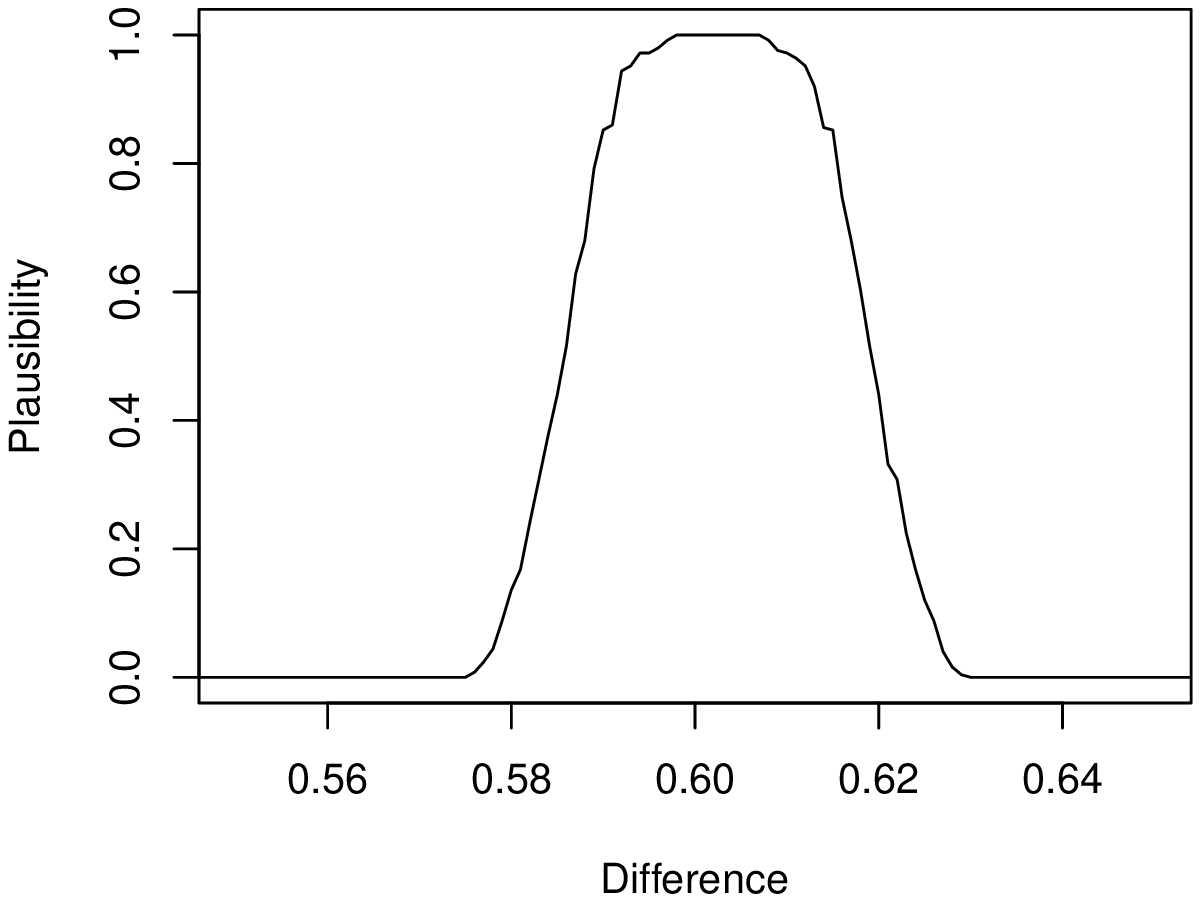}}}
\end{center}
\caption{In Panel~(a), the red curve is the plausibility contour for $\phi = \mathcal{V}(d)$ under the static regime, $d(X) \equiv 1$, blue is for the covariate-dependent regime \eqref{eq:not.static}, and black is the optimal regime discussed in Section~\ref{ss:optregime}. Panel (b) shows the plausibility contour for the difference between the covariate-dependent and static regimes.
}
\label{fig:dynreg}
\end{figure}

\subsection{Optimal regime}
\label{ss:optregime}

Estimating the value of a specific decision rule, $d$, may be of some interest in applications, but a more challenging problem is to identify the optimal regime within a given set $\mathcal{D}$, that is, a regime that leads to the best benefit on average if used to select treatment in the population. For situations where larger values of the response variable mean greater benefit to the patient, like in the example in Section~\ref{sss:valueregime2} above, the optimal regime $d^*$ is defined as one that leads to the maximum value among all $d \in \mathcal{D}$, i.e.,
\[d^* = \arg\max_{d \in \mathcal{D}} \mathcal{V}(d) \]
or, equivalently,
\begin{equation}
\label{eq:optregimedef}
    E\{Y^*(d^*)\} \geq E\{Y^*(d)\} \quad \text{for all }d \in \mathcal{D}.
\end{equation}
It is possible that more than one regime satisfies \eqref{eq:optregimedef}, but we will not concern ourselves with this technicality here. 
Recall that our focus is on a single stage treatment regime with $\mathcal{A} = \{0,1\}$.  For such a case, \citet{dynamicregimes} prove that \eqref{eq:optregimedef} is satisfied by
\begin{align}\label{eq:optregimey*}
d^*(x) &= \arg\max_{a \in \mathcal{A}} E\{Y^*(a) \mid X=x\} \nonumber \\
&= 1[E\{Y^*(1) \mid X=x\} > E\{Y^*(0) \mid X=x\}].
\end{align}




Just like in the previous subsection, uncertainty quantification about $d^*$ based on observed data $z^n$ requires us to rewrite \eqref{eq:optregimey*} in terms of $Z^n$.  Under the assumptions stated at the end of Section~\ref{ss:backgroundregime}, \eqref{eq:optregimey*} can be equivalently written as
\begin{equation}
\label{eq:optregimey**}
d^*(x) = \arg\max_{a \in \mathcal{A}} E\{Y \mid X=x, A=a\},
\end{equation}
which, under the outcome regression model formulation can itself be rewritten as 
\begin{align}
\label{eq:optQ}
d^*(x) = \arg\max_{a \in \mathcal{A}} Q_{x,a}(\theta) = 1\{Q_{x,1}(\theta) > Q_{x,0}(\theta)\}.
\end{align}
Moreover, the value of this optimal treatment is given by
\begin{equation}\label{eq:valueoptregime}
\mathcal{V}(d^*) = E\Bigl\{\max_{a \in \mathcal{A}} Q_{X,a}(\theta)\Bigr\},
\end{equation}
where, again, the outer expectation is with respect to the distribution of $X$.  The right-hand side of the above display is, again, a function $\phi=\phi(\theta)$ of $\theta$, so if we have a generalized IM for $\theta$, then we can readily obtain a marginal IM for $\phi$. 

As an illustration, consider again the example explored in Section~\ref{sss:valueregime2}, where $Q_{x,a}(\theta)$ is linear as in \eqref{eq:Qlinear}. In this case, it is clear that \eqref{eq:optQ} becomes
\[d^*(x) = 1\{\theta_3 + \theta_4 x_1 + \theta_5 x_2 > 0\}.\]
From \eqref{eq:valueoptregime}, the value of this optimal regime is given by
\begin{equation}
\label{eq:dstarex}
\mathcal{V}(d^*) = E\{\theta_0 + \theta_1X_1 + \theta_2X_2 + (\theta_3 + \theta_4X_1 + \theta_5X_2)1\{\theta_3 + \theta_4X_1 + \theta_5X_2 > 0\} \}.
\end{equation}
Uncertainty quantification about the value $\mathcal{V}(d^*)$ above is obtained through marginalization, as it was the case for the fixed regimes considered earlier. For example, to obtain the plausibility contour in Figure~\ref{fig:dynreg}(a), one starts with the generalized IM for $\theta$ and marginalize to the corresponding $\mathcal{V}(d^*)$ in \eqref{eq:dstarex}. Note how the value of $d^*$ is significantly greater than the values of the two fixed regimes also shown in Figure~\ref{fig:dynreg}.

\section{Conclusion}
\label{S:Conclusion}

Here we focused on direct, data-driven uncertainty quantification for unknowns defined as risk minimizers or solutions to estimating equations rather than parameters of a statistical model.  We presented a new generalized IM that not only avoids the explicit description of the data generating process, but does not require a statistical model at all. We showed that this construction 
leads to approximately valid uncertainty quantification in the sense of Theorem~\ref{thm:validity}.  This provides guarantees beyond those from classical confidence regions. That is, the IM's validity property applies to belief assignments to all assertions about the inferential target---even marginal inference about features of the original inferential target.  To our knowledge, this is the first paper providing direct and valid probabilistic uncertainty quantification in this practically relevant class of learning problems.  

Applications in cases beyond the simple, low-dimensional problems above will be reported elsewhere.  Of course, larger dimension creates computational challenges, so getting marginal plausibility contours for each component of the high-dimensional inferential target in an efficient way remains an open question. Since evaluation of $\hat\pi_{z^n}^\text{boot}$ is based on sampling (bootstrap or Monte Carlo), and marginalization is optimization in this imprecise probability setting, techniques like {\em stochastic approximation} or {\em stochastic gradient descent} seem especially promising; see, e.g., \citet{syring.martin.isipta21}.

We end this section with a brief discussion of some open questions.  First, it is well-known that the bootstrap often enjoys a certain higher-order accuracy, that is, the coverage probability of bootstrap-based confidence regions converge to the nominal level at a rate faster than the expected root-$n$; see, e.g., \citet{hall.book} and \citet{lehmann.large-sample}.  Similarly, in other settings, with simpler versions of the generalized IM framework developed here, it was observed empirically that the uniform limit distribution approximation for ``$\hat\pi_{Z^n}(\theta)$'' was quite accurate, even for small samples, suggesting some higher-order accuracy.  The proposed generalized IM in this paper borrows aspects of these two approaches that (at least empirically) enjoy higher-order accuracy.  Then the question is if this combination of two higher-order accurate methods is also higher-order accurate?

Second, although this did not appear in our illustrations in Section~\ref{S:dtr}, an especially challenging aspect of the dynamic treatment regime problem is non-regularity, resulting from the non-smooth ``max'' operator in \eqref{eq:optregimey**}, that affects the limit distribution theory of the M/Z-estimators and, in turn, the corresponding bootstrap-based inference.  A common remedy for failure of the bootstrap, e.g., due to non-regularity, is to ``under-sample'' with the so-called {\em m-out-of-n bootstrap}.  That is, \citet{bickel1997} showed that, by taking bootstrap samples of size $m=o(n)$, bootstrap failure could be avoided.  A more sophisticated $m$-out-of-$n$ bootstrap scheme was proposed in \citet{laber.mnboot} that could prevent bootstrap failure in the dynamic treatment regime setting resulting from the non-regularity induced by the ``max'' operator.  An interesting follow-up project could investigate the reliability of the proposed generalized IM equipped with a $m$-out-of-$n$ bootstrap strategy under non-regularity.  


\section*{Acknowledgments}
The authors thank the Guest Editor and the two anonymous reviewers for their helpful feedback on a previous version of this manuscript.  This work is partially supported by the U.S.~National Science Foundation, grants DMS--1811802 and SES--2051225.

\bibliographystyle{apalike}
\bibliography{literature}

\newcommand{\noop}[1]{}
\begin{thebibliography}{}

\bibitem[Balch et~al., 2019]{Ryansatellite}
Balch, M.~S., Martin, R., and Ferson, S. (2019).
\newblock Satellite conjunction analysis and the false confidence theorem.
\newblock {\em Proceedings of the Royal Society A: Mathematical, Physical and
  Engineering Sciences}, 475(2227):20180565.

\bibitem[Becker et~al., 2014]{Oja2013}
Becker, C., Fried, R., and Kuhnt, S. (2014).
\newblock {\em Robustness and Complex Data Structures: Festschrift in Honour of
  Ursula Gather}.
\newblock SpringerLink : B{\"u}cher. Springer Berlin Heidelberg.

\bibitem[Bhattacharya and Ghosal, 2022]{bhattacharya.ghosal.2022}
Bhattacharya, I. and Ghosal, S. (2022).
\newblock Bayesian inference on multivariate medians and quantiles.
\newblock {\em Statistica Sinica}, 32(1):517--538.

\bibitem[Bickel and Freedman, 1981]{bickelfreedman81}
Bickel, P.~J. and Freedman, D.~A. (1981).
\newblock Some asymptotic theory for the bootstrap.
\newblock {\em The Annals of Statistics}, 9(6):1196--1217.

\bibitem[Bickel et~al., 1997]{bickel1997}
Bickel, P.~J., G{\"o}tze, F., and van Zwet, W.~R. (1997).
\newblock Resampling fewer than {$n$} observations: gains, losses, and remedies
  for losses.
\newblock {\em Statistica Sinica}, 7(1):1--31.

\bibitem[Boos and Stefanski, 2013]{boos.stefanski.2013}
Boos, D.~D. and Stefanski, L.~A. (2013).
\newblock {\em Essential {S}tatistical {I}nference}.
\newblock Springer Texts in Statistics. Springer, New York.

\bibitem[Cahoon and Martin, 2020]{cahoon2019generalized1}
Cahoon, J. and Martin, R. (2020).
\newblock Generalized inferential models for meta-analyses based on few
  studies.
\newblock {\em Statistics and Applications}, 18(2):299--316.

\bibitem[Cahoon and Martin, 2021]{cahoon2019generalized2}
Cahoon, J. and Martin, R. (2021).
\newblock Generalized inferential models for censored data.
\newblock {\em International Journal of Approximate Reasoning}, 137:51--66.

\bibitem[Cella and Martin, 2021a]{CellaMartinBelief2021}
Cella, L. and Martin, R. (2021a).
\newblock Approximately valid and model-free possibilistic inference.
\newblock In Den{\oe}ux, T., Lef{\`e}vre, E., Liu, Z., and Pichon, F., editors,
  {\em Belief Functions: Theory and Applications}, pages 127--136, Cham.
  Springer International Publishing.

\bibitem[Cella and Martin, 2021b]{CellaMartinISIPTA21}
Cella, L. and Martin, R. (2021b).
\newblock Valid inferential models for prediction in supervised learning
  problems.
\newblock In Cano, A., De~Bock, J., Miranda, E., and Moral, S., editors, {\em
  Proceedings of the Twelveth International Symposium on Imprecise Probability:
  Theories and Applications}, volume 147 of {\em Proceedings of Machine
  Learning Research}, pages 72--82. PMLR.
\newblock Extended version available at
  \url{https://researchers.one/articles/21.12.00002v1}.

\bibitem[Cella and Martin, 2022]{CellaMartinConformal}
Cella, L. and Martin, R. (2022).
\newblock Validity, consonant plausibility measures, and conformal prediction.
\newblock {\em International Journal of Approximate Reasoning}, 141:110--130.

\bibitem[Chakraborty et~al., 2013]{laber.mnboot}
Chakraborty, B., Laber, E.~B., and Zhao, Y. (2013).
\newblock Inference for optimal dynamic treatment regimes using an adaptive
  {$m$}-out-of-{$n$} bootstrap scheme.
\newblock {\em Biometrics}, 69(3):714--723.

\bibitem[Chakraborty and Murphy, 2014]{chakradynreg}
Chakraborty, B. and Murphy, S.~A. (2014).
\newblock Dynamic treatment regimes.
\newblock {\em Annual Review of Statistics and Its Application}, 1(1):447--464.

\bibitem[Chatterjee and Bose, 2005]{ChatterjeeBose2005}
Chatterjee, S. and Bose, A. (2005).
\newblock {Generalized bootstrap for estimating equations}.
\newblock {\em The Annals of Statistics}, 33(1):414 -- 436.

\bibitem[Cheng and Huang, 2010]{chenghuang2010}
Cheng, G. and Huang, J. (2010).
\newblock Bootstrap consistency for general semiparametric {M}-estimation.
\newblock {\em The Annals of Statistics}, 38(5):2884--2915.

\bibitem[Choquet, 1954]{choquet1953}
Choquet, G. (1953--1954).
\newblock Theory of capacities.
\newblock {\em Universit\'e de Grenoble. Annales de l'Institut Fourier},
  5:131--295 (1955).

\bibitem[Conover, 1971]{conover1971practical}
Conover, W. (1971).
\newblock {\em Practical Nonparametric Statistics}.
\newblock Wiley.

\bibitem[Cui and Hannig, 2019]{cuihanning2019}
Cui, Y. and Hannig, J. (2019).
\newblock {Nonparametric generalized fiducial inference for survival functions
  under censoring}.
\newblock {\em Biometrika}, 106(3):501--518.

\bibitem[Cunen et~al., 2020]{cunen2020}
Cunen, C., Hjort, N.~L., and Schweder, T. (2020).
\newblock Confidence in confidence distributions!
\newblock {\em Proceedings of the Royal Society A: Mathematical, Physical and
  Engineering Sciences}, 476(2237):20190781.

\bibitem[Dawid and Stone, 1982]{dawidstone1982}
Dawid, A.~P. and Stone, M. (1982).
\newblock {The functional-model basis of fiducial inference}.
\newblock {\em The Annals of Statistics}, 10(4):1054 -- 1067.

\bibitem[Dempster, 1968]{dempster1968a}
Dempster, A.~P. (1968).
\newblock A generalization of {B}ayesian inference. ({W}ith discussion).
\newblock {\em Journal of the Royal Statistical Society, Series B},
  30:205--247.

\bibitem[Dempster, 2008]{DEMPSTER2008365}
Dempster, A.~P. (2008).
\newblock The {D}empster--{S}hafer calculus for statisticians.
\newblock {\em International Journal of Approximate Reasoning}, 48(2):365 --
  377.

\bibitem[Dempster, 2014]{dempster.copss}
Dempster, A.~P. (2014).
\newblock Statistical inference from a {D}empster--{S}hafer perspective.
\newblock In Lin, X., Genest, C., Banks, D.~L., Molenberghs, G., Scott, D.~W.,
  and Wang, J.-L., editors, {\em Past, Present, and Future of Statistical
  Science}, chapter~24. Chapman \& Hall/CRC Press.

\bibitem[Den{\oe}ux, 2009]{denoeux2009}
Den{\oe}ux, T. (2009).
\newblock Extending stochastic ordering to belief functions on the real line.
\newblock {\em Information Sciences}, 179(9):1362--1376.

\bibitem[Den{\oe}ux, 2014]{denoeux2014}
Den{\oe}ux, T. (2014).
\newblock Likelihood-based belief function: justification and some extensions
  to low-quality data.
\newblock {\em International Journal of Approximate Reasoning},
  55(7):1535--1547.

\bibitem[Dubois and Prade, 1986]{dubois.prade.1986}
Dubois, D. and Prade, H. (1986).
\newblock The principle of minimum specificity as a basis for evidential
  reasoning.
\newblock In {\em Proceedings of International Conference on Information
  Processing and Management of Uncertainty in Knowledge-based Systems}, pages
  75--84. Springer.

\bibitem[Dubois and Prade, 1988]{dubois.prade.book}
Dubois, D. and Prade, H. (1988).
\newblock {\em Possibility {T}heory}.
\newblock Plenum Press, New York.

\bibitem[Efron, 1979]{efron1979}
Efron, B. (1979).
\newblock Bootstrap methods: Another look at the jackknife.
\newblock {\em The Annals of Statistics}, 7(1):1--26.

\bibitem[Fisher, 1935]{fisherfiducial}
Fisher, R.~A. (1935).
\newblock The fiducial argument in statistical inference.
\newblock {\em Annals of Eugenics}, 6(4):391--398.

\bibitem[Fisher, 1973]{fisher1973}
Fisher, R.~A. (1973).
\newblock {\em Statistical Methods and Scientific Inference}.
\newblock Hafner Press, New York, 3rd edition.

\bibitem[Fraser, 1968]{fraser1968structure}
Fraser, D. A.~S. (1968).
\newblock {\em The Structure of Inference}.
\newblock John Wiley \& Sons Inc., New York.

\bibitem[Fraser, 2011]{Fraser2011}
Fraser, D. A.~S. (2011).
\newblock Is {B}ayes posterior just quick and dirty confidence.
\newblock {\em Statistical Science}, 26:299--316.

\bibitem[Ghosal and van~der Vaart, 2017]{ghosal.vaart.book}
Ghosal, S. and van~der Vaart, A. (2017).
\newblock {\em Fundamentals of Nonparametric {B}ayesian Inference}, volume~44
  of {\em Cambridge Series in Statistical and Probabilistic Mathematics}.
\newblock Cambridge University Press, Cambridge.

\bibitem[Ghosh and Ramamoorthi, 2003]{ghosh2003bayesian}
Ghosh, J. and Ramamoorthi, R. (2003).
\newblock {\em Bayesian Nonparametrics}.
\newblock Springer Series in Statistics. Springer.

\bibitem[Godambe, 1991]{godambe1991estimating}
Godambe, V. (1991).
\newblock {\em Estimating Functions}.
\newblock Oxford Science Publications. Clarendon Press.

\bibitem[Hahn, 1995]{hahn1995}
Hahn, J. (1995).
\newblock Bootstrapping quantile regression estimators.
\newblock {\em Econometric Theory}, 11(1):105--121.

\bibitem[Hall, 1992]{hall.book}
Hall, P. (1992).
\newblock {\em The {B}ootstrap and {E}dgeworth {E}xpansion}.
\newblock Springer Series in Statistics. Springer-Verlag, New York.

\bibitem[Hannig et~al., 2016]{MainHaning}
Hannig, J., Iyer, H., Lai, R. C.~S., and Lee, T. C.~M. (2016).
\newblock Generalized fiducial inference: A review and new results.
\newblock {\em Journal of the American Statistical Association},
  111(515):1346--1361.

\bibitem[Hjort et~al., 2010]{hhmw2010}
Hjort, N.~L., Holmes, C.~C., M\"uller, P., and Walker, S.~G., editors (2010).
\newblock {\em Bayesian {N}onparametrics}.
\newblock Cambridge University Press.

\bibitem[Hose and Hanss, 2021]{HOSE2021133}
Hose, D. and Hanss, M. (2021).
\newblock A universal approach to imprecise probabilities in possibility
  theory.
\newblock {\em International Journal of Approximate Reasoning}, 133:133--158.

\bibitem[Huber, 1981]{huber1981robust}
Huber, P. (1981).
\newblock {\em Robust Statistics}.
\newblock Wiley Series in Probability and Statistics. Wiley.

\bibitem[Huber, 1964]{huber1964}
Huber, P.~J. (1964).
\newblock {Robust estimation of a location parameter}.
\newblock {\em The Annals of Mathematical Statistics}, 35(1):73 -- 101.

\bibitem[Koenker, 2005]{koenker2005}
Koenker, R. (2005).
\newblock {\em Quantile {R}egression}.
\newblock Cambridge University Press, Cambridge.

\bibitem[Koenker and Bassett, 1978]{KoenkerHallock1978}
Koenker, R. and Bassett, G. (1978).
\newblock Regression quantiles.
\newblock {\em Econometrica}, 46(1):33--50.

\bibitem[Kosorok, 2008]{kosorok2008introduction}
Kosorok, M. (2008).
\newblock {\em Introduction to Empirical Processes and Semiparametric
  Inference}.
\newblock Springer Series in Statistics. Springer New York.

\bibitem[Lehmann, 1999]{lehmann.large-sample}
Lehmann, E.~L. (1999).
\newblock {\em Elements of Large-Sample Theory}.
\newblock Springer Texts in Statistics. Springer-Verlag, New York.

\bibitem[Liu and Martin, 2021]{imposs}
Liu, C. and Martin, R. (2021).
\newblock Inferential models and possibility measures.
\newblock {\em Handbook of Bayesian, Fiducial, and Frequentist Inference}, to
  appear; {\tt arXiv:2008.06874}.

\bibitem[Martin, 2015]{martin2015}
Martin, R. (2015).
\newblock Plausibility functions and exact frequentist inference.
\newblock {\em Journal of the American Statistical Association},
  110(512):1552--1561.

\bibitem[Martin, 2018]{MARTIN2018105}
Martin, R. (2018).
\newblock On an inferential model construction using generalized associations.
\newblock {\em Journal of Statistical Planning and Inference}, 195:105--115.

\bibitem[Martin, 2019]{MARTIN2019IJAR}
Martin, R. (2019).
\newblock False confidence, non-additive beliefs, and valid statistical
  inference.
\newblock {\em International Journal of Approximate Reasoning}, 113:39--73.

\bibitem[Martin, 2021]{imprecisefrequentist}
Martin, R. (2021).
\newblock An imprecise-probabilistic characterization of frequentist
  statistical inference.
\newblock {\em Researchers.One},
  \url{https://researchers.one/articles/21.01.00002}.

\bibitem[Martin, 2022]{martin.partial}
Martin, R. (2022).
\newblock Valid and efficient imprecise-probabilistic inference across a
  spectrum of partial prior information.
\newblock \url{https://researchers.one/articles/21.05.00001}.

\bibitem[Martin et~al., 2021]{ryanresponsecunen}
Martin, R., Balch, M.~S., and Ferson, S. (2021).
\newblock Response to the comment confidence in confidence distributions!
\newblock {\em Proceedings of the Royal Society A: Mathematical, Physical and
  Engineering Sciences}, 477(2250):20200579.

\bibitem[Martin and Liu, 2013]{mainMartin}
Martin, R. and Liu, C. (2013).
\newblock Inferential models: A framework for prior-free posterior
  probabilistic inference.
\newblock {\em Journal of the American Statistical Association}, 108:301--313.

\bibitem[Martin and Liu, 2015]{martinbook}
Martin, R. and Liu, C. (2015).
\newblock {\em Inferential Models: Reasoning with Uncertainty}.
\newblock Monographs in Statistics and Applied Probability Series. Chapman \&
  Hall/CRC Press.

\bibitem[Martin and Syring, 2022]{martin.syring.chapter2022}
Martin, R. and Syring, N. (2022).
\newblock Direct {G}ibbs posterior inference on risk minimizers: construction,
  concentration, and calibration.
\newblock In {\em {H}andbook of {S}tatistics: {A}dvancements in {B}ayesian
  {M}ethods and {I}mplementation}, to appear; {\tt arXiv:2203.09381}.

\bibitem[Molchanov, 2005]{molchanov2005}
Molchanov, I. (2005).
\newblock {\em Theory of Random Sets}.
\newblock Probability and Its Applications (New York). Springer-Verlag London
  Ltd., London.

\bibitem[Murphy et~al., 2001]{murphy2021}
Murphy, S.~A., van~der Laan, M.~J., Robins, J.~M., and Group, C. P. P.~R.
  (2001).
\newblock Marginal mean models for dynamic regimes.
\newblock {\em Journal of the American Statistical Association},
  96(456):1410--1423.
\newblock PMID: 20019887.

\bibitem[Nguyen, 2006]{nguyen2006introduction}
Nguyen, H.~T. (2006).
\newblock {\em An Introduction to {R}andom {S}ets}.
\newblock Chapman \& Hall/CRC, Boca Raton, FL.

\bibitem[Reid and Cox, 2015]{ReidandCox2015}
Reid, N. and Cox, D.~R. (2015).
\newblock On some principles of statistical inference.
\newblock {\em International Statistical Review}, 83(2):293--308.

\bibitem[Rubin, 1974]{rubin74}
Rubin, D.~B. (1974).
\newblock Estimating causal effects of treatments in randomized and
  nonrandomized studies.
\newblock {\em Journal of Educational Psychology.}, 66(5):688--701.

\bibitem[Rubin, 2005]{rubin2005}
Rubin, D.~B. (2005).
\newblock Causal inference using potential outcomes.
\newblock {\em Journal of the American Statistical Association},
  100(469):322--331.

\bibitem[Schweder and Hjort, 2016]{schweder_hjort_2016}
Schweder, T. and Hjort, N.~L. (2016).
\newblock {\em Confidence, Likelihood, Probability: Statistical Inference with
  Confidence Distributions}.
\newblock Cambridge Series in Statistical and Probabilistic Mathematics.
  Cambridge University Press.

\bibitem[Shafer, 1976]{shafer1976mathematical}
Shafer, G. (1976).
\newblock {\em A Mathematical Theory of Evidence}.
\newblock Princeton University Press, Princeton, N.J.

\bibitem[Shao and Tu, 1995]{shao.tu.book}
Shao, J. and Tu, D.~S. (1995).
\newblock {\em The {J}ackknife and {B}ootstrap}.
\newblock Springer Series in Statistics. Springer-Verlag, New York.

\bibitem[Singh, 1981]{singh1981}
Singh, K. (1981).
\newblock On the asymptotic accuracy of {E}fron's bootstrap.
\newblock {\em The Annals of Statistics}, 9(6):1187--1195.

\bibitem[Syring and Martin, 2021]{syring.martin.isipta21}
Syring, N. and Martin, R. (2021).
\newblock Stochastic optimization for numerical evaluation of imprecise
  probabilities.
\newblock In Cano, A., De~Bock, J., Miranda, E., and Moral, S., editors, {\em
  Proceedings of the Twelveth International Symposium on Imprecise Probability:
  Theories and Applications}, volume 147 of {\em Proceedings of Machine
  Learning Research}, pages 289--298. PMLR.

\bibitem[Tsiatis et~al., 2020]{dynamicregimes}
Tsiatis, A.~A., Davidian, M., Holloway, S.~T., and Laber, E.~B. (2020).
\newblock {\em Dynamic Treatment Regimes: Statistical Methods for Precision
  Medicine}.
\newblock Chapman \& Hall/CRC.

\bibitem[van~der Vaart and Wellner, 1996]{vaartwellner1996}
van~der Vaart, A.~W. and Wellner, J.~A. (1996).
\newblock {\em Weak {C}onvergence and {E}mpirical {P}rocesses}.
\newblock Springer-Verlag, New York.

\bibitem[Walley, 2002]{Walley2002ReconcilingFP}
Walley, P. (2002).
\newblock Reconciling frequentist properties with the likelihood principle.
\newblock {\em Journal of Statistical Planning and Inference}, 105:35--65.

\bibitem[Wasserman, 2006]{wasserman2006all}
Wasserman, L. (2006).
\newblock {\em All of Nonparametric Statistics}.
\newblock Springer Texts in Statistics. Springer New York.

\bibitem[Wellner and Zhang, 1996]{wellner.zhang.1996}
Wellner, J.~A. and Zhang, Y. (1996).
\newblock Bootstrapping {Z}-estimators.
\newblock Technical Report 308, University of Washington.
\newblock
  \url{https://stat.uw.edu/research/tech-reports/bootstrapping-z-estimators}.

\bibitem[Xie and Singh, 2013]{mainconfdist}
Xie, M.-g. and Singh, K. (2013).
\newblock Confidence distribution, the frequentist distribution estimator of a
  parameter: A review.
\newblock {\em International Statistical Review}, 81(1):3--39.

\bibitem[Zadeh, 1975]{Zadeh1975}
Zadeh, L.~A. (1975).
\newblock Fuzzy logic and approximate reasoning.
\newblock {\em Synthese}, 30(3-4):407--428.

\end{thebibliography}

\end{document}